\documentclass[11pt,a4paper]{amsart}
\usepackage[T1]{fontenc}

\usepackage[margin=3cm]{geometry}
\usepackage{amssymb,mathtools}
\usepackage{microtype,enumitem,xcolor,verbatim}

\usepackage[hidelinks]{hyperref}

\usepackage[looser]{newpxtext}
\usepackage[smallerops]{newpxmath} 

\usepackage{etoolbox}
\makeatletter
\patchcmd{\@settitle}
    {\uppercasenonmath\@title}   
    {}
    {}{\typeout{failed}}
\patchcmd{\@setauthors}
    {\MakeUppercase{\authors}}  
    {\authors}
    {}{\typeout{failed}}
\makeatother

\numberwithin{equation}{section}

\newtheorem{theorem}{Theorem}[section]
\newtheorem{proposition}[theorem]{Proposition}
\newtheorem{lemma}[theorem]{Lemma}
\newtheorem{corollary}[theorem]{Corollary}
\newtheorem{question}[theorem]{Question}
\newtheorem{maintheorem}{Theorem}

\theoremstyle{definition}

\theoremstyle{remark}
\newtheorem{remark}[theorem]{Remark}

\newcommand{\supp}{\operatorname{supp}}
\newcommand{\Occ}{\operatorname{Occ}}
\newcommand{\Int}{\operatorname{Int}}
\newcommand{\suppplus}{\supp_+}
\newcommand{\Fr}{\mathsf{Fr}} 
\newcommand{\cl}{\overline}
\newcommand{\one}{\mathbf{1}}

\title[Return Time Sets of Minimal Points]{Return Time Sets of Minimal Points for Countable Group Actions}

\author[J. Li]{Jian Li}
\address[J. Li]{Institute for  Mathematical Sciences and Artificial Intelligence \& Department of Mathematics,
	Shantou University, Shantou, 515821, Guangdong, China}
\email{lijian09@mail.ustc.edu.cn}
\urladdr{https://orcid.org/0000-0002-8724-3050}

\author[Y. Yang]{Yini Yang}
\address[Y. Yang]{Department of Mathematics,
	Shantou University, Shantou, 515821, Guangdong, China}
\email{ynyangchs@foxmail.com}
\urladdr{https://orcid.org/0000-0001-6564-2213}


\date{\today}

\begin{document}

\begin{abstract}
For a countably infinite discrete group $G$, we study subsets of $G$ that contain a neighbourhood return-time set of a minimal point in a $G$-flow $X$.
For $e\in F\subseteq G$, we characterize this property by the existence of $e\in A\subseteq F$ such that $\operatorname{Int}_K(A):=\{t\in G:Kt\subseteq A\}$ is right syndetic for every finite $K\subseteq A$, 
and equivalently by the support of a minimal binary configuration in the Bernoulli shift $2^G$.
Additional conditions on $\operatorname{Int}_K(A) \operatorname{Int}_K(A)^{-1}$ characterize realization by minimal flows that are weakly mixing of all finite orders or strongly mixing: for every such $K$, 
the former requires that for each finite $H\subseteq G$ there is $c\in G$ with $HcH\subseteq \operatorname{Int}_K(A) \operatorname{Int}_K(A) ^{-1}$, 
whereas the latter requires that $\operatorname{Int}_K(A)\operatorname{Int}_K(A)^{-1}$ be cofinite.
We also construct, for every compact metrizable minimal $G$-flow, a free minimal, strongly mixing Cantor flow disjoint from it. 
\end{abstract}

\maketitle

\section{Introduction}\label{sec:introduction}
Throughout the paper, $G$ denotes a countably infinite discrete group with identity $e$.  
A $G$-flow is a compact Hausdorff space $X$ equipped with a left action of $G$ by homeomorphisms.  For $x\in X$ and open set $U\subseteq X$, we write
\[
N_G(x,U):=\{g\in G:gx\in U\} 
\]
for the corresponding return-time set.  
If $x$ is a minimal point and $U$ is a neighbourhood of $x$, then $N_G(x,U)$ is right syndetic, meaning that $G=CN_G(x,U)$ for some finite $C\subseteq G$.  
Thus minimality imposes a concrete combinatorial constraint on every neighbourhood return-time set.

This observation leads to the inverse problem studied in this paper.  
Given $F\subseteq G$ with $e\in F$, when does $F$ contain a set of the form $N_G(x,U)$ for a minimal point $x$ and a neighbourhood $U$ of $x$?  We say that $F$ is dynamically central syndetic if there exist a minimal $G$-flow $X$, 
a point $x\in X$, and an open neighbourhood $U$ of $x$ such that $N_G(x,U)\subseteq F$.  
The adjective central records the pointed condition $x\in U$, and hence $e\in N_G(x,U)$.  
The realization problem concerns containment, $N_G(x,U)\subseteq F$, rather than equality with the prescribed set.

The inverse problem is closely connected with topological disjointness.  
Furstenberg introduced the notion of disjointness in topological dynamics and proved that every totally transitive cascade with a dense set of periodic points is disjoint from every minimal cascade~\cite{F67}.
Huang and Ye subsequently showed that every transitive cascade disjoint from all minimal cascades is weakly mixing and has a dense set of minimal points~\cite{HY05}. 
Dong, Shao, and Ye further established that every thickly syndetic subset of $\mathbb{N}$ contains the neighborhood return-time set of a point in a weakly mixing minimal subshift~\cite{DSY12}. 
These results highlight the close connections among disjointness, recurrence, and weak mixing, as well as their role in the symbolic realization of prescribed return-time behavior.

A direct antecedent is the work of Glasscock and Le on dynamically syndetic sets~\cite{GL24}, where cascade return-time supersets were characterized by finite-pattern and filter conditions. 
In the present paper, finite-pattern recursion controls recurrence of all occurring binary patterns, while the filter language records the resulting return-time largeness.  
This raises the question of which parts of that description persist for arbitrary countably infinite groups.  
In this direction, Xu and Ye proved that every right thickly syndetic subset of a countable group contains a non-empty subset whose indicator function is a minimal point of the Bernoulli shift $2^G$~\cite[Theorem~B]{XY24}. 
Glasner, Tsankov, Weiss, and Zucker showed that, for every infinite discrete group, the Bernoulli shift is disjoint from every minimal flow~\cite[Theorem 1.1]{GTWZ21}. 
More recently, Li, Liang, and Yang characterized quasi-central subsets of countably infinite discrete groups via return times, 
proving that $F\subseteq G$ is quasi-central if and only if it contains a return-time set of some piecewise syndetic recurrent point~\cite[Theorem 1.2]{LLY26}.
Our contribution is threefold: we prove the finite-interior and minimal-symbolic characterizations for arbitrary countably infinite groups, extend the construction to minimal flows with weak mixing of all finite orders or topological strong mixing, 
and construct a free strongly mixing Cantor partner disjoint from every prescribed compact metrizable minimal flow.  
Throughout, we keep the right-translation conventions explicit so that the statements remain valid in the non-abelian setting.

For $A\subseteq G$ and a finite set $K\subseteq G$, define the right $K$-interior of $A$ by
\[
\Int_K(A):=\{t\in G:Kt\subseteq A\}.
\]
If $A=N_G(x,U)$ with $x\in U$ in a minimal $G$-flow and $K\subseteq A$ is finite, then
\begin{equation}\label{eq:intro-interior_return_identity}
\Int_K(A)=N_G\!\biggl(x,\bigcap_{k\in K}k^{-1}U\biggr),
\end{equation}
so $\Int_K(A)$ is right syndetic.  Finite right interiors are therefore the natural positive-pattern invariant of the realization problem.

The converse requires more than recurrence of positive patterns.  
Right syndeticity of $\Int_K(A)$ controls finite patterns consisting only of $1$'s, 
whereas minimality of a binary configuration requires syndetic recurrence of every occurring finite pattern, including patterns containing both $1$'s and $0$'s.  
For finite $E_1\subseteq A$ and $E_2\subseteq G\setminus A$, the occurrence set of the pattern that is $1$ on $E_1$ and $0$ on $E_2$ is
\[
\bigcap_{a\in E_1}a^{-1}A
 \ \cap\!
 \bigcap_{b\in E_2}b^{-1}(G\setminus A).
\]
The construction below makes every finite pattern of the new binary configuration recur syndetically while keeping its positive support inside the prescribed set.  
This is the step that connects the finite-interior condition with the symmetric right-syndetic criterion for minimal indicators in~\cite[Proposition~3.10]{XY24}. 
We use the Bernoulli shift $2^G=\{0,1\}^G$ with left action $(gx)(h)=x(hg)$; for $A\subseteq G$, write $\one_A$ for its characteristic point.

Our basic characterization is the following.

\begin{maintheorem}\label{thm:equivalent_characterizations}
Let $F\subseteq G$ with $e\in F$. The following assertions are equivalent:
\begin{enumerate}[label=\textup{(\arabic*)}]
\item $F$ is dynamically central syndetic.
\item There exists $e\in A\subseteq F$ such that
$\Int_K(A)=\{t\in G:Kt\subseteq A\}$
is right syndetic for every finite $K\subseteq A$.
\item There exists $e\in A\subseteq F$ such that $\one_A$ is a minimal point of the Bernoulli shift $2^G$.
\end{enumerate}
\end{maintheorem}

The difficult direction is the passage from the finite-interior condition to a minimal binary configuration. Proposition~\ref{prop:binary_realization} isolates this step: its recursion converts positive placement information into permanent recurrence of mixed zero--one patterns. 
The other implications in Theorem~\ref{thm:equivalent_characterizations} follow from \eqref{eq:intro-interior_return_identity} for (1)$\Rightarrow$(2) and the identity cylinder for (3)$\Rightarrow$(1). 
Proposition~\ref{prop:more_equi_dyn_syn_sets} contains additional equivalent formulations using decreasing return-time sets, free Cantor actions, and syndetic idempotent filters; its free realization uses Proposition~\ref{prop:free_extension}.

The second theme concerns the combinatorial structure forced by mixing on neighbourhood return-time sets.
Recall that a $G$-flow $X$ is topologically weakly mixing of all finite orders if the diagonal action on $X^k$ is topologically transitive for every $k\geq 2$, and topologically strongly mixing if 
\[
N_G(U,V):=\{g\in G:gU\cap V\neq\varnothing\} 
\]
is cofinite for any non-empty open $U,V\subset X$. 
If $x\in X$ is a transitive point and $V\subset X$ is non-empty and open, then
\[
N_G(V,V)=N_G(x,V) N_G(x,V)^{-1}.
\]
Topologically weak mixing of all finite orders has a natural combinatorial signature.
A set $S\subseteq G$ is called bitranslation thick if for every finite $H\subseteq G$ there is $c\in G$ such that
\[
 HcH\subseteq S.
\]
For a topologically transitive $G$-flow, 
it is topologically weakly mixing of all finite orders if and only if $N_G(V,V)$ is bitranslation thick for every non-empty open set $V\subseteq X$; see Lemma~\ref{lem:allorders_self_hitting}.
Indeed, testing a finite family of rectangles produces a finite set of left and right translates, and the common translating element is exactly the centre $c$ in a set $HcH$.

We characterize the return-time sets of minimal $G$-flows that are weakly mixing of all finite orders as follows.
\begin{maintheorem}\label{thm:wm_allorders}
Let $F\subseteq G$ and $e\in F$. The following assertions are equivalent.
\begin{enumerate}
\item There exists a minimal $G$-flow $X$ that is topologically weakly mixing of all finite orders, a point $x\in X$, and an open neighbourhood $U$ of $x$ such that $N_G(x,U)\subseteq F$.
\item There exists $e\in A\subseteq F$ such that $\Int_K(A)$ is right syndetic and $\Int_K(A)\Int_K(A)^{-1}$ is bitranslation thick for every finite $K\subseteq A$.
\item There exists $e\in A\subseteq F$ such that $\one_A$ is a minimal point of $2^G$, and $\overline{G\one_A}$ is topologically weakly mixing of all finite orders.
\end{enumerate}
\end{maintheorem}

We characterize the return-time sets of topologically strongly mixing minimal $G$-flows as follows.

\begin{maintheorem}\label{thm:sm_sets}
Let $F\subseteq G$ and $e\in F$. 
The following three assertions are equivalent.
\begin{enumerate}
\item There exists a topologically strongly mixing minimal $G$-flow $X$, a point $x\in X$, and an open neighbourhood $U$ of $x$ such that $N_G(x,U)\subseteq F$.
\item There exists $e\in A\subseteq F$ such that $\Int_K(A)$ is right syndetic and $\Int_K(A)\Int_K(A)^{-1}$ is cofinite for every finite
$K\subseteq A$.
\item There exists $e\in A\subseteq F$ such that $\one_A$ is a minimal point of $2^G$ and $\overline{G\one_A}$ is topologically strongly mixing.
\end{enumerate}
\end{maintheorem}

For compact $G$-flows $X$ and $Y$, a joining is a non-empty closed invariant subset of $X\times Y$ projecting onto both factors; 
the flows are disjoint, written $X\perp Y$, when their only joining is $X\times Y$. 
Glasner, Tsankov, Weiss, and Zucker proved that every compact metrizable minimal $G$-flow admits a free minimal $G$-flow disjoint from it~\cite[Theorem 1.2]{GTWZ21}. 
To obtain mixing realizations by free actions on Cantor spaces, we prove the following strengthening, which is also of independent interest.

\begin{maintheorem}\label{thm:free_strong_mixing_disjoint}
Let $G$ be a countably infinite discrete group and let $X$ be a compact metrizable minimal $G$-flow. 
Then there exists a free topologically strongly mixing minimal $G$-flow on a Cantor space $Y$ such that $X\perp Y$.
\end{maintheorem}

This paper is organized as follows. 
Section~\ref{sec:preliminaries} fixes the group, action, return-time, and pattern conventions and records the extension tools used later.  
Section~\ref{sec:proof} develops the finite-pattern recursion and proves Theorem~\ref{thm:equivalent_characterizations}, then derives the finite-deletion and two-piece consequences.
Section~\ref{sec:mixing_characterizations} proves Theorems~\ref{thm:wm_allorders} and~\ref{thm:sm_sets} by adding bitranslation-thick or cofinite bridge conditions, and then constructs free mixing realizations and disjoint partners.  
Section~\ref{sec:disjointness_applications} applies these results to dynamically syndetic and thick sets, pointwise recurrence, and disjointness.

\section{Preliminaries}\label{sec:preliminaries}

We fix the right-translation conventions for subsets of $G$ and return-time sets, 
then recall finite binary patterns and the extension tools used below.

\subsection{Subsets of groups}
Let $G$ be a countably infinite discrete group with identity $e$.
For subsets $E,H\subseteq G$, write $EH=\{ab\colon a\in E,\ b\in H\}$ and $E^{-1}=\{a^{-1}\colon a\in E\}$. 
Set multiplication will always be used in the displayed order.  
A subset $A$ of $G$ is called \emph{right syndetic} if there is a finite set $C\subset G$ with $G=CA$,
and \emph{right thick} if for every finite $K\subset G$, there exists $g\in G$ such that $Kg\subset A$.
A superset of a right syndetic set is right syndetic.
A cofinite subset $L$ of $G$ is right syndetic:
choose $\ell_0\in L$ and put
\[
 C=\{e\}\cup\{g\ell_0^{-1}\colon g\in G\setminus L\}.
\]
Then $C$ is finite and $G=CL$.
The definitions imply that a set $A\subseteq G$ is right syndetic if and only if $A\cap B\neq\varnothing$ for every right thick $B\subseteq G$. Equivalently, $A$ is not right syndetic if and only if $G\setminus A$ is right thick.

For a finite set $K\subseteq G$ and $A\subseteq G$,
following~\cite[Section~2]{CCK14} we define the right $K$-interior of $A$ as
\[
 \Int_K(A):=\{t\in G\colon Kt\subseteq A\}.
\]

A subset $A$ of $G$ is \emph{right thickly syndetic} if
$\Int_K(A)$ is right syndetic for every finite $K\subseteq G$, 
and \emph{right piecewise syndetic} if $K^{-1}A$ is right thick for some finite $K\subseteq G$.  
The definitions imply that a set $A\subseteq G$ is right piecewise syndetic if and only if $A\cap B\neq\varnothing$ for every right thickly syndetic $B\subseteq G$, see e.g.~\cite[Lemma~2.2]{XY24}.

\begin{lemma}\label{lem:packing}
Let $D\subseteq G$ be finite and non-empty, and let $\Lambda\subseteq G$ be finite. 
Suppose $M\subseteq\Lambda$ is inclusion-maximal subject to the sets $Dt$, $t\in M$, being pairwise disjoint. 
Then
\[
 \Lambda\subseteq D^{-1}DM.
\]
In particular, if $D\subseteq S=S^{-1}$, then $\Lambda\subseteq S^2M$.
\end{lemma}

\begin{proof}
For $\ell\in\Lambda\setminus M$, maximality produces $t\in M$ with $D\ell\cap Dt\ne\varnothing$. 
Write $d\ell=d't$ with $d,d'\in D$.
Then $\ell=d^{-1}d't\in D^{-1}Dt$. 
For $\ell\in M$, use $e\in D^{-1}D$. 
This also covers the case $\Lambda=\varnothing$.
\end{proof}

For a family $\mathcal A$ of subsets of $G$, its upward closure is
\[
\mathcal{A}^{\uparrow}:=
 \{B\subseteq G:\text{there is }A\in\mathcal A\text{ with }A\subseteq B\}.
\]
A  \emph{proper filter} on $G$ is a non-empty upward-closed family of subsets of $G$, closed under finite intersections and not containing $\varnothing$. 
A proper filter $\mathscr F$ is called \emph{syndetic} if every member of $\mathscr F$ is right syndetic.
For a family $\mathscr F$ of subsets of $G$, put
\[
 A-\mathscr F:=\{g\in G:g^{-1}A\in\mathscr F\}.
\]
For families $\mathscr F,\mathscr H$ of subsets of $G$, define
\[
 \mathscr F+\mathscr H  :=\{A\subseteq G:A-\mathscr H\in\mathscr F\}.
\]
Thus $A-\mathscr F$ is a subset of $G$, whereas $\mathscr F+\mathscr H$ is a family of subsets of $G$. A proper filter $\mathscr F$ is called \emph{idempotent} if
\[
 \mathscr F\subseteq\mathscr F+\mathscr F,
\]
equivalently,
\[
 A-\mathscr F\in\mathscr F
 \qquad\text{for every }A\in\mathscr F.
\]

\subsection{Topological dynamics}

A \emph{$G$-flow} is a non-empty compact Hausdorff space $X$ equipped with a continuous left action of $G$. 
We write $gx$ for the image of $x$ under $g$. 
A subset $Y\subseteq X$ is invariant if $GY=Y$. 
For a non-empty closed invariant subset $Y$, the restricted action on $Y$ is a \emph{subflow}. 
A $G$-flow $X$ is \emph{minimal} if every orbit is dense or equivalently, if $X$ has no proper subflows.
A point $x\in X$ is a \emph{transitive point} if its orbit $Gx$ is dense in $X$.
A point $x\in X$ is \emph{minimal} if $(\cl{Gx},G)$ is minimal.
For $x\in X$ and $U,V\subset X$, let 
\[
 N_G(x,U):=\{g\in G:gx\in U\}\quad \text{and} \quad N_G(U,V):=\{g\in G:gU\cap V\neq \varnothing\}.
\]

A topological space is \emph{zero-dimensional} if it has a basis of clopen sets.
A \emph{Cantor space} is a non-empty compact metrizable, zero-dimensional space with no isolated  points. 
A \emph{joining} of compact $G$-flows $X$ and $Y$ is a non-empty closed invariant subset of $X\times Y$ whose two coordinate projections are
onto. The flows $X$ and $Y$ are \emph{disjoint} if their only joining is $X\times Y$, denoted by $X \perp Y$.
The following lemma goes back to Gottschalk and Hedlund~\cite{GH55}.
\begin{lemma}\label{lem:minimality}
Let $X$ be a $G$-flow.
Then a point $x\in X$ is minimal if and only if $N_G(x,U)$ is right syndetic for every open neighbourhood $U$ of $x$.
Moreover, if $X$ is minimal, then for every $x\in X$ and non-empty open set $U\subset X$, $N_G(x,U)$ is right syndetic.
\end{lemma}

Right syndeticity of neighbourhood return-time sets does not imply thickness. 
The next lemma shows that this distinction persists in every non-trivial minimal flow.
\begin{lemma}\label{lem:nonthick_return}
Let $X$ be a non-trivial minimal $G$-flow.
For every $x\in X$, there exists a neighbourhood $U$ of $x$ such that $N_G(x,U)$ is not right thick.
\end{lemma}
\begin{proof}
Fix $x\in X$. Because $X$ is minimal, $\cl{Gx}=X$.
Since $X$ has more than one point, there exists $g\in G$ such that $gx\neq x$.
Pick disjoint open subsets $V,W\subset X$ with $x\in V$ and $gx\in W$.
Let $U=V\cap g^{-1}W$. 
Then $U$ is a neighbourhood of $x$ and $U\cap gU=\varnothing$.
Let $A=N_G(x,U)$ and $F=\{e,g^{-1}\}$.
If there exists some $h\in G$ such that $Fh=\{h,g^{-1}h\}\subset A$,
that is, $hx\in U$ and $g^{-1}hx\in U$. 
Then $hx\in U\cap gU$, that is a contradiction.
Hence, $A=N_G(x,U)$ is not right thick.
\end{proof}

\begin{lemma}\label{lem:self_hitting_difference}
Let $x$ be a transitive point of a $G$-flow $X$ and 
let $V\subseteq X$ be non-empty and open. 
Then
\begin{equation*}
N_G(V,V)=N_G(x,V)N_G(x,V)^{-1}.
\end{equation*}
\end{lemma}

\begin{proof}
If $g=ba^{-1}$ with $a,b\in N_G(x,V)$, then $ax,bx\in V$ and $g(ax)=bx\in V$, 
so $g\in N_G(V,V)$. 
Conversely, if $g\in N_G(V,V)$, the open set $V\cap g^{-1}V$ is non-empty. 
Transitivity produces $a\in G$ with $ax\in V\cap g^{-1}V$. 
Then $a,ga\in N_G(x,V)$ and
$g=(ga)a^{-1}\in N_G(x,V) N_G(x,V)^{-1}$. 
\end{proof}

Let $X$ be a $G$-flow and $U,V\subset X$. Then the definitions show that for any $a,b\in G$,
\begin{equation*}
   aN_G(U,V)b^{-1}=N_G(bU,aV).
\end{equation*}

A $G$-flow $X$ is \emph{topologically transitive} if for any two non-empty open subsets $U,V\subset X$, $N_G(U,V)\neq\varnothing$.
By Baire's category theorem in a metric topologically transitive flow $X$ the set of transitive points is a dense $G_\delta$ subset of $X$.
A $G$-flow $X$ is \emph{topologically weakly mixing of all finite orders} if, for every $m\ge2$, the diagonal action on $X^m$ is topologically transitive.
If $G$ is abelian, topological weak mixing of all finite orders is equivalent to standard weak mixing. We refer the reader to~\cite{G05}.
A $G$-flow $X$ is \emph{topologically strongly mixing} if for any two non-empty open subsets $U,V\subset X$, $N_G(U,V)$ is cofinite. 

\subsection{Finite patterns and their exact occurrences}

Equip $2^G=\{0,1\}^G$ with the product topology and the action
\begin{equation}\label{eq:shift}
 (gx)(h)=x(hg).
\end{equation}
This is a left action, since $g_1(g_2x)(h)=x(hg_1g_2)=((g_1g_2)x)(h)$.
Since $G$ is countable, $2^G$ is compact and metrizable. 
A \emph{subshift} is a non-empty closed invariant subset of $2^G$.
For $B\subseteq G$, its characteristic function $\one_B\in2^G$
is defined by $\one_B(g)=1$ if $g\in B$ and $\one_B(g)=0$ otherwise.
We write $B^c=G\setminus B$.  
For $x\in2^G$, define its positive support by
\[
 \suppplus(x):=\{g\in G:x(g)=1\}.
\]

Finite-pattern notation follows the standard symbolic-dynamics convention; see, for example,~\cite[Chapter~1]{CC23}. 
A finite binary pattern is a function  $p:D\to\{0,1\}$ on a finite set $D\subseteq G$. 
Its positive support is
\[
 \suppplus(p)=\{d\in D:p(d)=1\}.
\]
For another pattern $q:V\to\{0,1\}$, define the \emph{occurrences} of $p$ in $q$ by
\begin{equation*}
 \Occ_q(p)=\{t\in G:Dt\subseteq V,
                  \ q(dt)=p(d)\text{ for all }d\in D\}.
\end{equation*}
Define $\Occ_x(p)$ for $x\in2^G$ by the same formula with $V=G$. Thus an occurrence centre $t$ records agreement on the entire translated domain $Dt$, including both the zero and the one entries.
For a pattern $p$ on $D$, write $U[p]=\{z\in2^G:z|_D=p\}$ for the corresponding cylinder.
If $p=x|_D$, then
\[
N_G(x,U[p])=\Occ_x(p).
\]
The notion of a symmetrically right syndetic set was introduced in \cite[Definition~6.9]{KRS22}. 
Here, we adopt a slightly modified version of the definition, following \cite[Proposition~3.9]{XY24}.
A subset $B\subseteq G$ is \emph{symmetrically right syndetic} if
for all finite sets $E_1\subseteq B$ and $E_2\subseteq B^c$, the set
\[
 \bigcap_{a\in E_1}a^{-1}B\ \cap\
 \bigcap_{b\in E_2}b^{-1}B^c
\]
is right syndetic. Empty intersections are interpreted as $G$.

The following criterion is~\cite[Proposition~3.10]{XY24} for non-empty proper subsets. 
With the empty-intersection convention it also includes $\varnothing$ and $G$. 
We include the short proof for completeness.

\begin{lemma}\label{lem:symmetric}
A subset $B\subseteq G$ is symmetrically right syndetic if and only
if $\one_B$ is a minimal point in $2^G$.
\end{lemma}

\begin{proof}
For a finite set $D\subseteq G$, put $p=\one_B|_D$. 
Then
\[
 N_G(\one_B,U[p])=
 \bigcap_{a\in D\cap B}a^{-1}B\ \cap\
 \bigcap_{b\in D\cap B^c}b^{-1}B^c.
\]
These cylinders form a neighbourhood basis at $\one_B$. 
Conversely, any pair $E_1\subseteq B$, $E_2\subseteq B^c$ occurs in this formula with $D=E_1\cup E_2$. 
The assertion follows from
Lemma~\ref{lem:minimality}.
\end{proof}

An exact copy of $p$ placed on $Dt$ creates no $1$ outside $A$ if and only if $t\in\Int_{\suppplus(p)}(A)$. 
The zero entries have no support restriction, but they must remain unchanged if the occurrence is to
survive. 
We shall ensure this by placing copies on disjoint domains.

The following observation records the multiplication order for inherited occurrences. 
If $q$ occurs exactly on $Vt$ in a pattern $w$, and $u\in\Occ_q(p)$, then
\begin{equation}\label{eq:inheritance}
 ut\in\Occ_w(p).
\end{equation}
Indeed, $Dut\subseteq Vt$ and $w(dut)=q(du)=p(d)$ for each $d\in D$.

\begin{lemma}\label{lem:finite_covers}
Let $D_1\subseteq D_2\subseteq\cdots$ be finite sets with $\bigcup_rD_r=G$. 
Let $p_r\in\{0,1\}^{D_r}$ satisfy $p_{r+1}|_{D_r}=p_r$. 
Suppose that for each $i$ there is a finite $Q_i\subseteq G$ such that
\begin{equation}\label{eq:fixed_cover}
 D_r\subseteq Q_i\Occ_{p_r}(p_i)\qquad(r\ge i).
\end{equation}
Then the unique configuration $x$ extending every $p_r$ is minimal. 
If also $p_r(e)=1$ and $\suppplus(p_r)\subseteq A$ for every $r$,
then $x(e)=1$ and $\suppplus(x)\subseteq A$.
\end{lemma}

\begin{proof}
Fix $i$ and $g\in G$. Choose $r\ge i$ with $g\in D_r$. 
By \eqref{eq:fixed_cover}, write $g=qt$ with $q\in Q_i$ and $t\in\Occ_{p_r}(p_i)$. 
Since $x$ extends $p_r$, this is also an exact occurrence in $x$. 
Therefore
\[
 G=Q_i\Occ_x(p_i).
\]
The cylinders $U[p_i]$ form a neighbourhood basis at $x$. 
Their return-time sets are right syndetic, so $x$ is minimal.
The assertions about the support and the value at $e$ follow from nestedness and exhaustion of $G$.
\end{proof}

\subsection{Free actions}\label{sec:free}

For a point $x$ in a $G$-flow $X$, its stabilizer is
$\operatorname{Stab}_G(x)=\{g\in G:gx=x\}$.
An action is \emph{free} if every point has stabilizer $\{e\}$. 
A \emph{factor map} is a continuous surjection commuting with the actions; its domain is called an extension of its range. 

The following two lemmas are folklore; we provide their proofs for the sake of complement.
\begin{lemma}\label{lem:compact_free}
Every countable group $G$ admits a non-empty compact metrizable, zero-dimensional free $G$-flow.
\end{lemma}

\begin{proof}
For each $s\in G\setminus\{e\}$, let
\[
	Y_s=\{c\in\{0,1,2\}^G:c(sh)\ne c(h)\text{ for every }h\in G\}.
\]
To see that $Y_s$ is non-empty, partition $G$ into right cosets of $\langle s\rangle$. 
On a coset $\langle s\rangle h$, join each 	$s^nh$ to $s^{n+1}h$. 
This graph is a bi-infinite path if $s$ has infinite order, 
an even cycle (or a single edge) if its order is even, and an odd cycle if its order is odd. 
In the first two cases use alternating colors $0,1$. 
For an odd cycle of length $m\ge3$,	color the first $m-1$ consecutive vertices alternately $0,1$ and the last vertex $2$. 
In every case adjacent vertices have different colors, producing an element of $Y_s$.
	
Each $Y_s$ is closed in $\{0,1,2\}^G$ and is invariant under $(gc)(h)=c(hg)$: for every $h\in G$,
\[
	(gc)(sh)=c(shg)\ne c(hg)=(gc)(h).
\]
Let $Y=\prod_{s\ne e}Y_s$, with the diagonal action.
It is non-empty, compact, metrizable and zero-dimensional since the index set and $G$ are countable. 
If $g\ne e$ and	$y=(c_s)_{s\ne e}\in Y$, its coordinate indexed by $s=g$ satisfies
\[
	(gc_g)(e)=c_g(g)\ne c_g(e).
\]
Therefore $gy\ne y$, and the action on $Y$ is free. 
\end{proof}

\begin{proposition}\label{prop:free_extension}
For every minimal $G$-flow $X$, there exists a free minimal $G$-flow $Z$ and a factor map $\pi:Z\to X$. 
If $X$ is metrizable, then $Z$ can be chosen metrizable.
If $X$ is zero-dimensional, then $Z$ can be chosen zero-dimensional.
For every $x\in X$, every $z\in\pi^{-1}(x)$, and every open neighbourhood $U$ of $x$, one has
\begin{equation}\label{eq:return_lift}
N_G(z,\pi^{-1}(U))=N_G(x,U).
\end{equation}
\end{proposition}

\begin{proof}
Take the free $G$-flow $Y$ in Lemma~\ref{lem:compact_free} and consider the diagonal action on $X\times Y$. 
The family of its non-empty closed invariant subsets has a minimal member $Z$ under inclusion: 
a chain has non-empty intersection by compactness, 
and Zorn's lemma, applied with reverse inclusion, shows such a member.
For every $z\in Z$, the orbit closure $\cl{Gz}$ is a non-empty closed invariant subset of $Z$, consequently equals $Z$. 
Thus $Z$ is minimal. It is compact Hausdorff, and the action on it is free because the $Y$ coordinate is free. 
If $X$ is metrizable, then $X\times Y$ and consequently its closed subspace $Z$ are metrizable.
The restricted first-coordinate projection $\pi:Z\to X$ is continuous and equivariant. 
Its image is non-empty, compact and invariant; since $X$ is Hausdorff, $\pi(Z)$ is closed. 
Thus minimality of $X$ implies $\pi(Z)=X$.
If $X$ is zero-dimensional, the product $X\times Y$ and its subspace $Z$ are zero-dimensional. 
Finally, for $\pi(z)=x$,
\[
gz\in\pi^{-1}(U) \quad\Longleftrightarrow\quad
\pi(gz)=g\pi(z)=gx\in U,
\]
which establishes~\eqref{eq:return_lift}. 
Continuity of $\pi$ establishes the assertions concerning open and clopen neighbourhoods.
\end{proof}

\section{Equivalent characterizations of dynamically central syndetic sets}\label{sec:proof}

This section develops the finite-pattern construction that produces a minimal binary realization inside a prescribed set $F$, 
and then derives the equivalent characterizations stated in 
Theorem~\ref{thm:equivalent_characterizations}.

\subsection{Word geometry within subgroup cosets}

The finite-set geometry below is used only inside the subgroup generated by a fixed finite symmetric set; see~\cite{L17,AHO20} for related word-metric constructions.

Let $S=S^{-1}\subseteq G$ be finite and contain $e$. 
Define $S^0=\{e\}$, and let $S^k$ denote the product of $k$ copies of $S$.
The sequence $(S^k)_{k\ge0}$ is increasing. 
In $H=\langle S\rangle$, write
\[
|h|_S=\min\{k\in\mathbb{N}_0\colon h\in S^k\}.
\]
On $G$ define the extended distance
\begin{equation*}
	d_S(a,b)=
	\begin{cases}
		|ab^{-1}|_S,&ab^{-1}\in H,\\
		+\infty,&ab^{-1}\notin H.
	\end{cases}
\end{equation*}
On each right coset $Hb$, this is the graph distance whose edges join $u$ to $su$ for $s\in S\setminus\{e\}$. 
In particular, points in the same coset are joined by a shortest path of integer length.
The distance is invariant under simultaneous right translation of both arguments. 
For a point $a\in G$ and a non-empty finite set $V\subseteq G$, put
\[
\delta_S(a,V)=\min_{v\in V}d_S(a,v).
\]
For $k\in\mathbb{N}_0$ we have
\begin{equation}\label{eq:distance_products}
	S^kV=\{a\in G\colon \delta_S(a,V)\le k\}.
\end{equation}

\begin{lemma}\label{lem:geodesic}
Let $S=S^{-1}\subset G$ be finite with $e\in S$, and let $a\in G$, $V\subset G$ be non-empty finite.
If $k:=\delta_S(a,V)<\infty$, then there is a path
\[
a=a_0,a_1,\ldots,a_k=v\in V
\]
with $d_S(a_j,a_{j+1})=1$ and $\delta_S(a_j,V)=k-j$. 
Moreover, if $d_S(a,b)\le m<\infty$, then $\delta_S(b,V)<\infty$ and
\[
	|\delta_S(a,V)-\delta_S(b,V)|\le m.
\]
\end{lemma}

\begin{proof}
Choose $v\in V$ with $d_S(a,v)=k$ and take a shortest word path from $a$ to $v$ in $Hv$. 
The remaining part of the path has $\delta_S(a_j,V)\le k-j$. 
A strict inequality would produce a path from $a$ to $V$ shorter than $k$, a contradiction. 
The last assertion follows from the triangle inequality, first using $a$ and then $b$ as the intermediate point.
\end{proof}

\begin{lemma}\label{lem:buffer}
Let $S=S^{-1}\subset G$ be finite with $e\in S$, and let $D,C\subseteq S$ with $e\in D$. 
Let $L\subseteq G$ satisfy $G=CL$. 
For any non-empty finite $V\subseteq G$ and any
\[
	a\in S^RV\setminus S^2V,\qquad R\ge3,
\]
where $R$ is an integer, there is $\ell\in L$ such that
\begin{equation*}
	D\ell\subseteq S^5V\setminus V,
	\qquad a\in S^{R-2}\ell.
\end{equation*}
\end{lemma}

\begin{proof}
By \eqref{eq:distance_products}, the integer $k=\delta_S(a,V)$ satisfies $3\le k\le R$. 
Along the path supplied by Lemma~\ref{lem:geodesic}, choose $z$ with
\[
	\delta_S(z,V)=3,\qquad d_S(a,z)=k-3.
\]
Write $z=c\ell$ with $c\in C$ and $\ell\in L$. 
Then $d_S(z,\ell)\le1$. 
If $y=d\ell\in D\ell$, then
\[
	d_S(y,z)=|dc^{-1}|_S\le2.
\]
Lemma~\ref{lem:geodesic} produces $1\le\delta_S(y,V)\le5$. 
Thus $D\ell\subseteq S^5V\setminus V$.
Also $d_S(a,\ell)\le k-2\le R-2$, which is equivalent to $a\in S^{R-2}\ell$.
\end{proof}

\subsection{A binary realization proposition}
The next proposition is the construction that supplies the minimal binary point required in Theorem~\ref{thm:equivalent_characterizations}(2)$\implies$(3).
Its induction recursively inserts exact copies of the previously constructed finite patterns and makes their occurrences permanently syndetic, while keeping all positive coordinates inside the prescribed set.

\begin{proposition}\label{prop:binary_realization}
Let $F\subseteq G$ with $e\in F$. 
If  $\Int_K(F)$ is right syndetic for every finite $K\subseteq F$, 
then there is a minimal point $x\in 2^G$ such that
\[
x(e)=1 \quad\text{and}\quad\suppplus(x)\subseteq F.
\]
In particular, $F$ is dynamically central syndetic.
\end{proposition}

\begin{proof}
Enumerate $G=\{g_1,g_2,\ldots\}$ with $g_1=e$.
We construct a sequence of patterns by induction. At each stage, the domains grow to exhaust $G$, the support condition keeps every positive coordinate inside the prescribed set, and the covering invariants make every earlier pattern recur syndetically in all later domains.

\noindent \textbf{Step 1: Initial construction.} 
Start with
\[
P_0=B_1=\{e\}, \quad D_1=P_0B_1
\]
and 
\[
p_1\colon D_1\to\{0,1\},\quad p_1(e)=1.
\]
Then $K_1:=\suppplus(p_1)=\{e\}\subseteq F$. 
Let $L_1=\Int_{K_1}(F)$.
Choose a finite $C_1\ni e$ with $G=C_1L_1$, 
and choose a finite symmetric $S_1$ containing $D_1\cup C_1$.

\noindent\textbf{Step 2: Induction hypothesis.}
Assume that the finite sets $B_i,D_i,C_i,S_i$ and the patterns
$p_i\in\{0,1\}^{D_i}$ have been constructed for $1\le i\le r$.
For $2\le i\le r$, put
\[
P_{i-1}=S_1^5S_2^5\cdots S_{i-1}^5,
\qquad
B_i=D_{i-1}\cup\{g_1,\dotsc,g_i\},
\]
and require $D_i=P_{i-1}B_i$ for $1\le i\le r$, with $P_0$
as fixed in Step~1.

The patterns are compatible and retain the prescribed positive support:
\[
p_{j+1}|_{D_j}=p_j \quad(1\le j<r),
\qquad
p_i(e)=1,\quad \suppplus(p_i)\subseteq F
\quad(1\le i\le r).
\]
For each $1\le i\le r$, set
\[
K_i:=\suppplus(p_i),
\qquad
L_i:=\Int_{K_i}(F).
\]
The finite set $C_i$ contains $e$ and satisfies $G=C_iL_i$,
while $S_i$ is finite and symmetric and contains $D_i\cup C_i$.

For $1\le i<j\le r$, define
\[
V_i^{(j)}=S_i^5S_{i+1}^5\cdots S_{j-1}^5B_j,
\qquad
V_j^{(j)}=B_j,
\]
and define the level-$i$ frame of $D_j$ by
\[
\Fr_{i,j}:=D_j\setminus S_i^3V_{i+1}^{(j)}.
\]
Assume the frame covering invariant
\[
\Fr_{i,j}\subseteq S_i^6\Occ_{p_j}(p_i)
\qquad(1\le i<j\le r),
\]
and the global covering invariant
\[
D_j\subseteq S_i^8\Occ_{p_j}(p_i)
\qquad(1\le i\le j\le r).
\]
The inclusion $D_1\subseteq S_1^8\Occ_{p_1}(p_1)$ holds
trivially, and no frame is needed at equal indices.

\textbf{Step 3: Inductive extension.}  Set
\begin{equation*}
  P_{r}=S_1^5S_2^5\cdots S_{r}^5, \quad B_{r+1}=D_r\cup\{g_1,\ldots,g_{r+1}\},
 \quad D_{r+1}=P_rB_{r+1}.
\end{equation*}
For $1\le i\le r+1$, write $V_i=V_i^{(r+1)}$. 
Thus
\[
 V_{r+1}=B_{r+1},\qquad V_i=S_i^5V_{i+1}
 \quad(1\le i\le r),\qquad V_1=D_{r+1}.
\]
For $1\le i\le r$, the product geometry implies
\[
 V_i=S_i^5V_{i+1}\subseteq D_{r+1}\subseteq S_i^6V_{i+1}.
\]
Moreover, if $b\in D_{r+1}$, $a\notin D_{r+1}$, and $a\in S_i^2b$, then
$b\in\Fr_{i,r+1}$: otherwise $b\in S_i^3V_{i+1}$ and
$a\in S_i^5V_{i+1}=V_i\subseteq D_{r+1}$.

Now we define a pattern $p_{r+1}$ on $D_{r+1}$.
Keep the central copy of $p_r$ on $D_r$ unchanged. 
Process the levels in descending order $r,r-1,\ldots,1$. 
Just before processing level $i$, 
let $\mathcal H_i$ be the family consisting of the central
domain $D_r$ and all external domains $D_kt$ with $k>i$ already placed at this stage. 
The external level-$k$ domains will be required to lie in $V_k$. Since $D_r\subseteq B_{r+1}$ and
$V_k\subseteq V_{i+1}$ for $k>i$, we have
\begin{equation}\label{eq:higher_core}
 \bigcup\mathcal H_i\subseteq V_{i+1}.
\end{equation}

Define the admissible centres at level $i$ by
\begin{equation*}
 \Lambda_i=
 \{\ell\in L_i:D_i\ell\subseteq V_i,
             \ D_i\ell\cap H=\varnothing
                       \text{ for every }H\in\mathcal H_i\}.
\end{equation*}
This set is finite, because $e\in D_i$ implies $\Lambda_i\subseteq V_i$. 
Choose an inclusion-maximal subset $M_i\subseteq\Lambda_i$ for which the sets $D_it$, $t\in M_i$, are pairwise disjoint. 
On each such domain write the translated pattern
\[
 p_{r+1}(dt)=p_i(d)\qquad(d\in D_i).
\]
These assignments do not conflict with any earlier assignment. 
At the next level, the newly written domains are included among the higher blocks. 
After processing level $1$, assign $0$ to every remaining point of $D_{r+1}$.

The resulting pattern extends $p_r$. 
Each newly assigned $1$ belongs to some $\suppplus(p_i)t\subseteq F$, since $t\in M_i\subseteq L_i$.
Therefore $\suppplus(p_{r+1})\subseteq F$ and $p_{r+1}(e)=1$. 
Lower-level assignments never enter a previously assigned domain. 
Thus each inserted copy remains exact, and
\begin{equation}\label{eq:selected_occurrences}
 M_i\subseteq\Occ_{p_{r+1}}(p_i).
\end{equation}
The same observation preserves all occurrences inherited inside any higher block, by \eqref{eq:inheritance}.

\medskip 
\noindent\textbf{Claim 1:} $\Fr_{i,r+1}\subseteq S_i^6\Occ_{p_{r+1}}(p_i)$ for  $1\le i \le r$.

\begin{proof}[Proof of Claim 1]
Fix $1\le i\le r$ and $a\in\Fr_{i,r+1}$. 
Since $D_{r+1}\subseteq S_i^6V_{i+1}$ and $\Fr_{i,r+1}=D_{r+1}\setminus S_i^3V_{i+1}$, we have
\[
 a\in S_i^6V_{i+1}\setminus S_i^3V_{i+1}.
\]
Apply Lemma~\ref{lem:buffer} with $R=6$ to $S_i,D_i,C_i,L_i$ and $V_{i+1}$.
The buffer lemma produces $\ell\in L_i$ such that
\[
 D_i\ell\subseteq S_i^5V_{i+1}\setminus V_{i+1} =V_i\setminus V_{i+1},
 \qquad a\in S_i^4\ell.
\]
By \eqref{eq:higher_core}, this domain avoids every higher block, so $\ell\in\Lambda_i$. 
Lemma~\ref{lem:packing} shows $\ell\in S_i^2t$ for some $t\in M_i$. 
Hence
\[
 a\in S_i^4S_i^2t=S_i^6t
   \subseteq S_i^6\Occ_{p_{r+1}}(p_i).
\]
This proves $\Fr_{i,r+1}\subseteq S_i^6\Occ_{p_{r+1}}(p_i)$.
\end{proof}

\medskip 
\noindent\textbf{Claim 2:} $D_{r+1}\subseteq S_i^8\Occ_{p_{r+1}}(p_i)$ for $1\le i \le r$.

\begin{proof}[Proof of Claim 2]
Fix $i\le r$ and $a\in D_{r+1}$. First test whether $a$ lies in a member of $\mathcal H_i$. If it does not, choose the candidate centre $\ell$ below and distinguish whether $D_i\ell$ leaves $V_i$, meets a member of $\mathcal H_i$ inside $V_i$, or lies in $V_i$ and avoids all members of $\mathcal H_i$. These alternatives lead to the following four cases.

\emph{Case 1: $a$ belongs to a higher block.}
Write $a=bt\in D_kt\in\mathcal H_i$. 
An external block has $k>i$; for the central block take $k=r$ and $t=e$, which also permits $k=i=r$. 
The induction hypothesis in Step~2 states that $D_k\subseteq S_i^8\Occ_{p_k}(p_i)$, so there exists $u\in\Occ_{p_k}(p_i)$ with $b\in S_i^8u$.
Since $D_kt$ carries an exact copy of $p_k$, the centre $ut$ is in $\Occ_{p_{r+1}}(p_i)$. 
Therefore
\[
 a\in S_i^8(ut).
\]

For the remaining cases, assume $a\notin\bigcup\mathcal H_i$.
Use $G=C_iL_i$ to write
\begin{equation}\label{eq:candidate_centre}
 a=c\ell,\qquad c\in C_i\subseteq S_i,\quad\ell\in L_i.
\end{equation}
Every $y=d\ell\in D_i\ell$ satisfies
\begin{equation}\label{eq:candidate_distance}
 ay^{-1}=cd^{-1}\in S_i^2,
 \qquad ya^{-1}=dc^{-1}\in S_i^2.
\end{equation}

\emph{Case 2: $D_i\ell$ is not contained in $V_i$.}
Choose $y\in D_i\ell\setminus V_i$. 
If $a\in S_i^3V_{i+1}$, then \eqref{eq:candidate_distance} would show $y\in S_i^5V_{i+1}=V_i$, a contradiction. 
Consequently $a\in\Fr_{i,r+1}$. 
The frame covering already proved produces
\[
 a\in S_i^6\Occ_{p_{r+1}}(p_i)
   \subseteq S_i^8\Occ_{p_{r+1}}(p_i).
\]

For the two remaining cases, assume $D_i\ell\subseteq V_i$.

\emph{Case 3: $D_i\ell\subseteq V_i$ and $D_i\ell$ meets $D_kt\in\mathcal H_i$.}
Choose
\[
 y=bt\in D_i\ell\cap D_kt,\qquad b\in D_k.
\]
Since $a$ lies outside every higher block, $at^{-1}\notin D_k$.
On the other hand, \eqref{eq:candidate_distance} shows
\begin{equation*}
 (at^{-1})b^{-1}=ay^{-1}\in S_i^2.
\end{equation*}
If $k>i$ and $b\notin\Fr_{i,k}$, then $b\in S_i^3V_{i+1}^{(k)}$; hence $at^{-1}\in S_i^5V_{i+1}^{(k)}=V_i^{(k)}\subseteq D_k$, a contradiction. Thus $b\in\Fr_{i,k}$.
Since $\Fr_{i,k}\subseteq S_i^6\Occ_{p_k}(p_i)$ by Step~2, choose
\[
 u\in\Occ_{p_k}(p_i),\qquad b\in S_i^6u.
\]
The inherited centre $ut$ belongs to $\Occ_{p_{r+1}}(p_i)$, and
\begin{equation*}
 a\in S_i^2y=S_i^2bt
   \subseteq S_i^2S_i^6ut=S_i^8ut.
\end{equation*}

The only possibility with $k=i$ is the central block at level $i=r$, so $t=e$. 
In that case $b\in D_i\subseteq S_i$ and \eqref{eq:candidate_distance} produces $a\in S_i^2b\subseteq S_i^3$.
The central occurrence $e\in\Occ_{p_{r+1}}(p_i)$ proves the desired covering. 
This case does not require a frame with equal indices.

\emph{Case 4: $D_i\ell$ is contained in $V_i$ and avoids every higher block.}
Then $\ell\in\Lambda_i$. 
By Lemma~\ref{lem:packing}, some $t\in M_i$ satisfies $\ell\in S_i^2t$. 
Using \eqref{eq:candidate_centre} and \eqref{eq:selected_occurrences}, we get
\[
 a\in S_i^3t\subseteq S_i^8\Occ_{p_{r+1}}(p_i).
\]

These cases exhaust $D_{r+1}$. 
Hence $D_{r+1}\subseteq S_i^8\Occ_{p_{r+1}}(p_i)$  holds for each $1\leq i\le r$.
\end{proof}

It remains to define the data needed for the next induction step.  
Since $p_{r+1}(e)=1$ and $\suppplus(p_{r+1})\subseteq F$, put
\[
 K_{r+1}:=\suppplus(p_{r+1}),
 \qquad L_{r+1}:=\Int_{K_{r+1}}(F).
\]
By the hypothesis on $F$, choose a finite set $C_{r+1}\ni e$ with $G=C_{r+1}L_{r+1}$.  
Choose a finite symmetric set $S_{r+1}\ni e$ such that
\[
D_{r+1}\cup C_{r+1}\subseteq S_{r+1}.
\]
Then $D_{r+1}\subseteq S_{r+1}^8\Occ_{p_{r+1}}(p_{r+1})$, because $D_{r+1}\subseteq S_{r+1}$ and
$e\in\Occ_{p_{r+1}}(p_{r+1})$.  
Thus all data required in the induction hypothesis are constructed at level $r+1$.

This completes the induction.

\noindent\textbf{Step 4: Passage to the limit and  Verification of the conclusion}
Since $g_r\in B_r\subseteq D_r$ for every $r$, the increasing domains $D_r$ exhaust $G$.  
The compatibility relations  $p_{r+1}|_{D_r}=p_r$ consequently define a unique configuration $x\in2^G$
with $x|_{D_r}=p_r$ for every $r$.
Applying Lemma~\ref{lem:finite_covers} with the fixed finite sets $Q_i=S_i^8$
proves that $x$ is minimal, $x(e)=1$, and $\suppplus(x)\subseteq F$.

Define
\[
 X=\cl{Gx},\qquad U=\{y\in X:y(e)=1\}.
\]
The set $U$ is clopen and contains $x$.  By \eqref{eq:shift},
\[
 g\in N_G(x,U)
 \quad\Longleftrightarrow\quad
 (gx)(e)=1
 \quad\Longleftrightarrow\quad
 x(g)=1.
\]
Consequently,
\[
 N_G(x,U)=\suppplus(x)\subseteq F,
\]
which implies that $F$ is dynamically central syndetic. 
\end{proof}

\subsection{Proof of Theorem~\ref{thm:equivalent_characterizations} and consequences}
\begin{proof}[Proof of Theorem~\ref{thm:equivalent_characterizations}]
(1)$\implies$(2). Choose a minimal $G$-flow $(X,G)$, a point $x\in X$,
and a neighbourhood $U\ni x$ such that
\[
 A:=N_G(x,U)\subseteq F.
\]
Then $e\in A$. For every finite $K\subseteq A$,
\begin{equation}\label{eq:interior_return_identity}
 \Int_K(A) =N_G\biggl(x,\bigcap_{k\in K}k^{-1}U\biggr).
\end{equation}
Indeed, for $t\in G$,
\[
\begin{aligned}
 t\in\Int_K(A)
 &\Longleftrightarrow kt\in N_G(x,U)\quad(k\in K)\\
 &\Longleftrightarrow (kt)x\in U\quad(k\in K)\\
 &\Longleftrightarrow tx\in k^{-1}U\quad(k\in K).
\end{aligned}
\]
Because $k\in A$ implies $kx\in U$, the intersection in \eqref{eq:interior_return_identity} is a neighbourhood of $x$. 
Minimality therefore makes its return-time set right syndetic.

(2)$\implies$(3). 
Apply Proposition~\ref{prop:binary_realization} to the witness $A$. 
The construction produces a minimal $z\in2^G$ with $z(e)=1$ and
$\suppplus(z)\subseteq A\subseteq F$. 
Put $B=\suppplus(z)$. 
Then $e\in B\subseteq F$ and $z=\one_B$, so (3) holds.

(3)$\implies$(1). In the orbit closure $X=\overline{G\one_A}$,
let $U=\{y\in X:y(e)=1\}$. 
This is a clopen neighbourhood of $\one_A$, and \eqref{eq:shift} shows that
\[
N_G(\one_A,U)=A\subseteq F.
\]
Thus $F$ is dynamically central syndetic.
\end{proof}

Applying Theorem~\ref{thm:equivalent_characterizations}, 
we obtain the following further equivalent characterizations of dynamically central syndetic sets.

\begin{proposition}\label{prop:more_equi_dyn_syn_sets}
Let $F\subseteq G$ with $e\in F$. The following five assertions are equivalent:
\begin{enumerate}[label=\textup{(\arabic*)}]
\item $F$ is dynamically central syndetic.
\item there exists a sequence of decreasing right syndetic sets
\[
F_1\supseteq F_2\supseteq\cdots,
\qquad F_n\subseteq F,
\]
such that for every $n$ and every finite $K\subseteq F_n$ there exists $m$ with $KF_m\subseteq F_n$.
\item There exists a minimal $G$-flow $(X,G)$, a point $x\in X$, and a non-empty
open set $U\subseteq X$ with $x\in\overline U$ such that
\[
N_G(x,U)\cup\{e\}\subseteq F.
\]
\item there exists a free minimal $G$-flow on a Cantor space $X$, a point $x\in X$, and a clopen neighbourhood
$U\ni x$ such that $N_G(x,U)\subseteq F$.
\item $F$ belongs to a syndetic, idempotent filter on $G$. 
\end{enumerate}
\end{proposition}
\begin{proof}
(1)$\implies$(2). Apply Theorem~\ref{thm:equivalent_characterizations}(3) to obtain a minimal binary point
$z$ with $z(e)=1$ and $\suppplus(z)\subseteq F$. 
Let $(U_n)$ be a decreasing clopen neighbourhood basis at $z$, with $U_1$ contained in the identity cylinder, and set $F_n=N_G(z,U_n)$. 
Then $F_n\subseteq F$ and each $F_n$ is right syndetic. 
If $K\subseteq F_n$ is finite, choose $m\ge n$ with
\[
 U_m\subseteq U_n\cap\bigcap_{k\in K}k^{-1}U_n.
\]
For $t\in F_m$ this produces $kt\in F_n$ for every $k\in K$, hence $KF_m\subseteq F_n$.

(2)$\implies$(1). Put $H_n=F_n\cup\{e\}$. 
The sets $H_n$ remain decreasing and right syndetic. They also retain the condition that for every $n$ and finite $K\subseteq H_n$ there is $m\ge n$ with $KH_m\subseteq H_n$. 
For every finite $K\subseteq H_1$, choose $m$ with $KH_m\subseteq H_1$. 
Then $H_m\subseteq\Int_K(H_1)$, so the finite-interior condition holds for $H_1\subseteq F$.
Thus, the assertion (2) of Theorem~\ref{thm:equivalent_characterizations} holds.

(1)$\implies$(3). This implication is immediate.

(3)$\implies$(1). Put $F_0=N_G(x,U)\cup\{e\}$. 
For finite $K\subseteq F_0$, let $H=K\setminus\{e\}$ and
\[
Q=\bigcap_{h\in H}h^{-1}U.
\]
This is a neighbourhood of $x$. 
Since $x\in\overline U$, the set $V=U\cap Q$ is non-empty and open, and $N_G(x,V)\subseteq\Int_K(F_0)$. 
Choose $g_0$ with $g_0x\in V$ and an open neighbourhood $W\ni x$ with $g_0W\subseteq V$.
Then $g_0N_G(x,W)\subseteq N_G(x,V)$. 
Minimality makes  $N_G(x,W)$ right syndetic, hence so is $N_G(x,V)$ and consequently $\Int_K(F_0)$.
Thus, the assertion (2) of Theorem~\ref{thm:equivalent_characterizations} holds.

(1)$\implies$(4). 
Use the minimal binary realization supplied by the assertion (3) of Theorem~\ref{thm:equivalent_characterizations}, and apply
Proposition~\ref{prop:free_extension}. 
The resulting extension is free, minimal, zero-dimensional, compact, and metrizable, and it preserves the return-time set. 
If it had an isolated point, minimality would make the whole space a finite discrete orbit, contradicting freeness and the infinitude of $G$. 
Hence it is a Cantor space. 
The identity cylinder pulls back to a clopen neighbourhood with return set contained in $F$.

(4)$\implies$(1). This implication is immediate.

(1)$\implies$(5). Let
\[
 \mathscr{F}_x:= \{N_G(x,V):V\text{ open and }x\in V\}^{\uparrow}.
\]
This is a filter because $N_G(x,V_1)\cap N_G(x,V_2)=N_G(x,V_1\cap V_2)$, and it is syndetic by minimality. 
If $A=N_G(x,V)$ and $g\in A$, then
\[
 g^{-1}A=N_G(x,g^{-1}V)\in\mathscr{F}_x.
\]
Consequently $A\subseteq A-\mathscr{F}_x$, 
so $A-\mathscr{F}_x\in\mathscr{F}_x$; upward closure implies
$\mathscr{F}_x\subseteq\mathscr{F}_x+\mathscr{F}_x$. 
Thus $F\in\mathscr{F}_x$.

(5)$\implies$(1). 
Let $\mathscr{F}$ be a syndetic, idempotent  filter containing $F$. 
Put
\[
 B:=F-\mathscr{F},\qquad A:=F\cap B.
\]
Since $F\in\mathscr{F}$, $e\in B$, consequently $e\in A$. If $g\in A$,
then $g^{-1}F\in\mathscr{F}$, and
\[
 g^{-1}B=(g^{-1}F)-\mathscr{F}\in\mathscr{F}
\]
by idempotence. 
Therefore $g^{-1}A\in\mathscr{F}$. 
For every finite $K\subseteq A$,
\[
 \Int_K(A)=\bigcap_{k\in K}k^{-1}A\in\mathscr{F},
\]
so this set is right syndetic. 
\end{proof}

\begin{corollary}\label{cor:finite_deletion_dcs}
Let $F\subseteq G$ be dynamically central syndetic and let
$K\subseteq G\setminus\{e\}$ be finite. 
Then $F\setminus K$ is dynamically central syndetic.
\end{corollary}

\begin{proof}
Choose $e\in A\subseteq F$ as in Theorem~\ref{thm:equivalent_characterizations}(2)
and set $A'=A\setminus K$. For finite $H\subseteq A'$,
\[
	\Int_H(A')	=\Int_H(A)\cap \bigcap_{h\in H}(G\setminus h^{-1}K).
\]
This deletes only finitely many points from the right syndetic set $\Int_H(A)$, 
and a right syndetic subset of an infinite 	group remains right syndetic after finite deletion. 
Consequently, $A'$ satisfies Theorem~\ref{thm:equivalent_characterizations}(2), proving the claim.
\end{proof}

\begin{corollary}\label{cor:two_piece_dcs}
If $F\subseteq G$ is dynamically central syndetic, 
then there are disjoint sets $F_1,F_2\subseteq F$ such that $F_1\cup\{e\}$ and
$F_2\cup\{e\}$ are dynamically central syndetic.
\end{corollary}

\begin{proof}
Use Proposition~\ref{prop:more_equi_dyn_syn_sets}(4) to choose a free minimal Cantor $G$-flow, $x\in X$, and a clopen $W\ni x$ with $N_G(x,W)\subseteq F$. 
Choose pairwise disjoint non-empty clopen sets $C_n\subseteq W\setminus\{x\}$ converging to $x$, and put
\[
	U_1=\bigcup_{n\ge1}C_{2n},\qquad
	U_2=\bigcup_{n\ge1}C_{2n-1},\qquad
	F_i=N_G(x,U_i).
\]
Then $F_1,F_2\subseteq F$, $F_1\cap F_2=\varnothing$, and $x\in\overline{U_i}$. 
Proposition~\ref{prop:more_equi_dyn_syn_sets}(3)
shows that $F_i\cup\{e\}$ is dynamically central syndetic for $i=1,2$.
\end{proof}

\section{Return-time sets of mixing minimal flows} \label{sec:mixing_characterizations}

Mixing imposes further conditions on the difference sets of neighbourhood return-time sets. 
The return-time sets themselves need not be right thick, even for a non-trivial minimal flow, as Lemma~\ref{lem:nonthick_return} shows. 
For topologically strongly mixing flows, self-hitting sets are cofinite; 
for flows that are weakly mixing of all finite orders, they satisfy a two-sided translation condition. 
The following two subsections show the corresponding return-time characterizations.

\subsection{Return-time sets of minimal flows weakly mixing of all finite orders}\label{subsec:all_orders}

A set $S\subseteq G$ is \emph{bitranslation thick} if for all finite $H\subseteq G$ there exists $c\in G$ such that
\begin{equation*}
HcH\subseteq S.
\end{equation*}
The property passes to supersets and is invariant under left and right translations. 
It expresses simultaneous largeness for finitely many prescribed left and right translates.
If $G$ is abelian, then bitranslation thickness is the standard thickness.

\begin{lemma}\label{lem:bt_finite_deletion}
If $S$ is bitranslation thick and $Q\subseteq G$ is finite, then $S\setminus Q$ is bitranslation thick.
\end{lemma}
\begin{proof}
Fix a finite $H\subseteq G$. 
A center $c$ can fail $HcH\subseteq S\setminus Q$ only if $c\in H^{-1}QH^{-1}$. 
Choose distinct $r_1,\ldots,r_N\in G$ with $N>|H^{-1}QH^{-1}|$, let $H'=\bigcup_jHr_j$, and choose $d$ with $H'dH\subseteq S$. 
Each $c_j=r_jd$ satisfies $Hc_jH\subseteq S$, and one $c_j$ avoids $H^{-1}QH^{-1}$. 
That $c_j$ works. 
\end{proof}

The following lemma translates the dynamical finite-order mixing condition into the subset condition used in the construction.
\begin{lemma}\label{lem:allorders_self_hitting}
Let $X$ be a $G$-flow.
Then the following assertions are equivalent:
\begin{enumerate}
    \item $X$ is weakly mixing of all finite orders.
    \item For every pair of non-empty open sets $U,V\subseteq X$, $N_G(U,V)$ is bitranslation thick.
    \item  $X$ is topologically transitive and $N_G(V,V)$ is bitranslation thick for every non-empty open $V\subseteq X$.
\end{enumerate}
\end{lemma}
\begin{proof}
(1)$\implies$(2).
Since $X^2$ is topologically transitive, $X$ is topologically transitive. 
Fix non-empty open $U,V\subseteq X$ and a finite non-empty $H\subseteq G$.
The transitivity of $X^{|H|^2}$ produces $c\in G$ such that
\[
c(bU)\cap a^{-1}V\neq\varnothing\qquad(a\in H,\ b\in H).
\]
This shows that $acb\in N_G(U,V)$ for all $a,b\in H$.
Hence $HcH\subseteq N_G(U,V)$.
For $H=\varnothing$, the required containment is automatic.

(2)$\implies$(3). This implication is immediate.

(3)$\implies$(2).
Fix non-empty open $U,V\subseteq X$. 
By transitivity choose $g\in G$ with $gU\cap V\neq\varnothing$, and choose non-empty open $W\subseteq U\cap g^{-1}V$. 
Then
\begin{equation}
  gN_G(W,W)\subseteq N_G(U,V).  \label{eq:W_W_U_V}
\end{equation}
For finite $H\subseteq G$, bitranslation thickness of $N_G(W,W)$ produces $c$ with $(g^{-1}H)cH\subseteq N_G(W,W)$. 
Multiplying by $g$ and using \eqref{eq:W_W_U_V}, $HcH\subseteq N_G(U,V)$. 
Thus $N_G(U,V)$ is bitranslation thick.

(2)$\implies$(1).
Fix $m\ge2$ and basic open rectangles $\mathcal U=\prod_{i=1}^mU_i$ and $\mathcal V=\prod_{i=1}^mV_i$ in $X^m$. 
By (2), $X$ is transitive.
Construct a non-empty open $W_0$ and $b_i\in G$ with $b_iW_0\subseteq U_i$ as follows. 
Start with non-empty open $O_0$. Inductively choose $b_i$ with $b_iO_{i-1}\cap U_i\neq\varnothing$, 
then choose non-empty open $O_i\subseteq O_{i-1}\cap b_i^{-1}U_i$. 
Put $W_0=O_m$.
Next construct $W\subseteq W_0$ and $a_i\in G$ with $a_i^{-1}W\subseteq V_i$. 
Starting with $W_0$, choose $a_i$ so that $a_iV_i\cap W_{i-1}\neq\varnothing$, 
and then choose non-empty open $W_i\subseteq W_{i-1}\cap a_iV_i$. Put $W=W_m$.
By (2), $S=N_G(W,W)$ is bitranslation thick.
Our target is a common element of $\bigcap_{i=1}^mN_G(U_i,V_i)$. Apply bitranslation thickness with the finite set $H=\{a_i,b_i:1\le i\le m\}$. There is $c\in G$ such that $a_icb_i\in S$ for all $i=1,2,\dotsc,m$.
Then
\[
c\in a_i^{-1}Sb_i^{-1}=N_G(b_iW,a_i^{-1}W)\subseteq N_G(U_i,V_i).
\]
Therefore $cU_i\cap V_i\neq\varnothing$ for all $i$, so $c\mathcal U\cap\mathcal V\neq\varnothing$. 
Thus $X^m$ is topologically transitive, for every $m\ge2$. 
\end{proof}

For $e\in A\subseteq G$ and finite $K\subseteq A$, define the bridge set
\begin{equation*}
\Gamma_A(K):=\{g\in G:\Int_{K\cup Kg}(A)\text{ is right syndetic}\}.
\end{equation*}
The elementary identity
\begin{equation*}
\Int_K(A)\cap g^{-1}\Int_K(A)=\Int_{K\cup Kg}(A)
\end{equation*}
implies
\begin{equation*}
\Gamma_A(K)=\{g\in G\colon \Int_K(A)\cap g^{-1}\Int_K(A)\text{ is right syndetic}\}.
\end{equation*}

\begin{lemma}\label{lem:interior_vs_product}
For $e\in A\subseteq G$, if $\Int_K(A)$ is right syndetic for every finite $K\subseteq A$,
then $\Gamma_A(K)=\Int_K(A)\Int_K(A)^{-1}$ for every finite $K\subseteq A$.
\end{lemma}

\begin{proof}
Fix a finite subset $K$ of $A$ and put $B=\Int_K(A)$.

For every $g\in\Gamma_A(K)$, the set $B\cap g^{-1}B$ is right syndetic and consequently non-empty. 
If $t\in B\cap g^{-1}B$, then $t,gt\in B$, so $g=(gt)t^{-1}\in BB^{-1}$.
Thus $\Gamma_A(K)\subseteq BB^{-1}$.

Conversely, if $g=ba^{-1}\in BB^{-1}$, put $L=Ka\cup Kb\subseteq A$. 
The set $C=\Int_L(A)$ is right syndetic. 
For $h\in C$, $t=ah$ satisfies $Kt\subseteq A$ and $Kgt\subseteq A$, so $aC\subseteq\Int_{K\cup Kg}(A)$.
Therefore $g\in\Gamma_A(K)$. 
Hence $BB^{-1}\subseteq\Gamma_A(K)$.
\end{proof}

\begin{proposition}\label{prop:allorders_realization}
Assume $e\in F\subset G$ and $\Gamma_F(K)$ is bitranslation thick for every finite $K\subseteq F$. 
Then there is a binary configuration $x\in2^G$ such that
\[
x(e)=1,\qquad \suppplus(x)\subseteq F,
\]
and $X=\overline{Gx}$ is minimal and weakly mixing of all finite orders.
\end{proposition}

\begin{proof}
Enumerate $G=\{g_1,g_2,\ldots\}$ with $g_1=e$.
For each $i\in\mathbb{N}$, let $H_i=\{g_1,g_2,\dotsc,g_i\}$.
For every finite $K\subseteq F$, the set $\Gamma_F(K)$ is non-empty.
Choose $\gamma\in\Gamma_F(K)$. 
Then $\Int_{K\cup K\gamma}(F)\subseteq\Int_K(F)$ is right syndetic;
therefore $\Int_K(F)$ is right syndetic.
We construct an increasing sequence of binary patterns. 
The buffer sets separate the domains of inserted copies, the frame and global covering invariants ensure minimality of the limit, and the bridge requirements enforce weak mixing of every finite order.

\medskip 
\noindent\textbf{Step 1: Initial construction.} 
Start with
\[
P_0=B_1=\{e\}, \quad D_1=P_0B_1
\]
and 
\[
p_1\colon D_1\to\{0,1\},\quad p_1(e)=1.
\]
Then $K_1:=\suppplus(p_1)=\{e\}\subseteq F$. 
Let $L_1=\Int_{K_1}(F)$.
Choose a finite $C_1\ni e$ with $G=C_1L_1$, 
and choose a finite symmetric $S_1$ containing $D_1\cup C_1$.

\medskip
\noindent\textbf{Step 2: Induction hypothesis.}
Assume that the finite sets $B_i,D_i,C_i,S_i$ and the patterns
$p_i\in\{0,1\}^{D_i}$ have been constructed for $1\le i\le r$.
The patterns are compatible and retain the prescribed positive support:
\[
p_{j+1}|_{D_j}=p_j \quad(1\le j<r),
\qquad
p_i(e)=1,\quad \suppplus(p_i)\subseteq F
\quad(1\le i\le r).
\]
For each $1\le i\le r$, set
\[
K_i:=\suppplus(p_i),
\qquad
L_i:=\Int_{K_i}(F).
\]
The finite set $C_i$ contains $e$ and satisfies $G=C_iL_i$,
while $S_i$ is finite and symmetric and contains $D_i\cup C_i$.

For each $2\le i\le r$, put
\[
P_{i-1}=S_1^5S_2^5\cdots S_{i-1}^5,
\qquad
Q_{i-1}=D_{i-1}^{-1}D_{i-1}.
\]
Assume that there is $c_{i-1}\in G$ such that
\[
H_{i-1}c_{i-1}H_{i-1}
\subseteq\Gamma_F(K_{i-1})\setminus Q_{i-1}.
\]
For each $\gamma\in H_{i-1}c_{i-1}H_{i-1}$, assume that a centre
$t_{i-1,\gamma}$ has been chosen with
\[
t_{i-1,\gamma}\in\Int_{K_{i-1}\cup K_{i-1}\gamma}(F).
\]
The domains $D_{i-1}t_{i-1,\gamma}$ and
$D_{i-1}\gamma t_{i-1,\gamma}$ are pairwise disjoint as $\gamma$ varies,
avoid $D_{i-1}$, and carry exact copies of $p_{i-1}$ in $p_i$.
Define
\[
\mathcal J_{i-1}
=\{D_{i-1}t_{i-1,\gamma},\ D_{i-1}\gamma t_{i-1,\gamma}
  :\gamma\in H_{i-1}c_{i-1}H_{i-1}\},
\]
and require
\[
B_i=D_{i-1}\cup\{g_1,\ldots,g_i\}\cup\bigcup_{J\in\mathcal J_{i-1}}J
\qquad(2\le i\le r).
\]
The domains satisfy $D_i=P_{i-1}B_i$ for $1\le i\le r$,
with $P_0$ and $B_1$ as fixed in Step~1.

For $1\le i<j\le r$, define
\[
V_i^{(j)}=S_i^5S_{i+1}^5\cdots S_{j-1}^5B_j,
\qquad
V_j^{(j)}=B_j,
\]
and define the level-$i$ frame of $D_j$ by
\[
\Fr_{i,j}:=D_j\setminus S_i^3V_{i+1}^{(j)}.
\]
Assume the frame covering invariant
\[
\Fr_{i,j}\subseteq S_i^6\Occ_{p_j}(p_i)
\qquad(1\le i<j\le r),
\]
and the global covering invariant
\[
D_j\subseteq S_i^8\Occ_{p_j}(p_i)
\qquad(1\le i\le j\le r).
\]
The inclusion $D_1\subseteq S_1^8\Occ_{p_1}(p_1)$ holds
trivially, and no frame is needed at equal indices.

\medskip
\noindent\textbf{Step 3: Inductive extension.}
Put $Q_r=D_r^{-1}D_r$. Since $Q_r$ is finite, Lemma~\ref{lem:bt_finite_deletion} and the
hypothesis imply that we may choose $c_r\in G$ such that
\[
H_rc_rH_r\subseteq \Gamma_F(K_r)\setminus Q_r.
\]
For each $\gamma\in H_rc_rH_r$, the set
$\Int_{K_r\cup K_r\gamma}(F)$ is right syndetic. Choose
$t_{r,\gamma}\in\Int_{K_r\cup K_r\gamma}(F)$ successively so that the sets
\[
D_rt_{r,\gamma},\qquad D_r\gamma t_{r,\gamma}
\]
are pairwise disjoint as $\gamma$ varies and avoid $D_r$.
Indeed, if $W$ is the union of $D_r$ and all previously selected bridge domains,
the centres forbidden by these collisions belong to the finite set $D_r^{-1}W\cup\gamma^{-1}D_r^{-1}W$.
The two domains for the same $\gamma$ are disjoint for every centre because $\gamma\notin D_r^{-1}D_r$. 
Since every right syndetic subset of the infinite group $G$ is infinite, a permitted centre exists. 
Define
\[
\mathcal J_r:=\{D_rt_{r,\gamma},\ D_r\gamma t_{r,\gamma}
       :\gamma\in H_rc_rH_r\}.
\]
Moreover, $t_{r,\gamma}\in\Int_{K_r\cup K_r\gamma}(F)$ guarantees that every positive coordinate in these copies lies in $F$.

Set
\[
P_r=S_1^5\cdots S_r^5,\qquad
B_{r+1}=D_r\cup\{g_1,\ldots,g_{r+1}\}\cup\bigcup_{J\in\mathcal J_r}J,
\qquad D_{r+1}=P_rB_{r+1}.
\]
For $1\le i\le r+1$, write $V_i=V_i^{(r+1)}$. 
Thus
\[
 V_{r+1}=B_{r+1},\qquad V_i=S_i^5V_{i+1}
 \quad(1\le i\le r),\qquad V_1=D_{r+1}.
\]

Since $e\in B_i$, we have $P_{i-1}\subseteq D_i\subseteq S_i$.
For $1\le i\le r$, the product geometry implies
\[
 V_i=S_i^5V_{i+1}\subseteq D_{r+1}\subseteq S_i^6V_{i+1}.
\]
Moreover, if $b\in D_{r+1}$, $a\notin D_{r+1}$, and $a\in S_i^2b$, then
$b\in\Fr_{i,r+1}$: otherwise $b\in S_i^3V_{i+1}$ and $a\in S_i^5V_{i+1}=V_i\subseteq D_{r+1}$.

Now we define a pattern $p_{r+1}$ on $D_{r+1}$.
Keep the central copy of $p_r$ on $D_r$ unchanged, and put exact copies of $p_r$ on each piece of $\mathcal{J}_{r}$.  
Process the levels in descending order $r,r-1,\ldots,1$. 
Just before processing level $i$, let $\mathcal H_i$ be the family consisting of the central domain $D_r$, 
every bridge domain in $\mathcal{J}_r$, and all external domains $D_kt$ with $k>i$ already  placed at this stage. 
The external level-$k$ domains will be required to lie in $V_k$. 
Since $D_r\subseteq B_{r+1}$ and $V_k\subseteq V_{i+1}$ for $k>i$, we have
\begin{equation}\label{eq:wm_higher_core}
 \bigcup\mathcal H_i\subseteq V_{i+1}.
\end{equation}

Define the admissible centres at level $i$ by
\begin{equation*}
 \Lambda_i=
 \{\ell\in L_i:D_i\ell\subseteq V_i,
             \ D_i\ell\cap H=\varnothing
                       \text{ for every }H\in\mathcal H_i\}.
\end{equation*}
This set is finite, because $e\in D_i$ implies $\Lambda_i\subseteq V_i$. 
Choose an inclusion-maximal subset $M_i\subseteq\Lambda_i$ for which the sets $D_it$, $t\in M_i$, are pairwise disjoint. 
Put exact copies of $p_i$ on each piece $D_it$, $t\in M_i$.  
After processing level $1$, assign $0$ to every remaining point of $D_{r+1}$.
The resulting $p_{r+1}$ extends $p_r$,  $\suppplus(p_{r+1})\subseteq F$ and $p_{r+1}(e)=1$.

Every lower-level assignment avoids the previously assigned domains.
Consequently each inserted copy remains exact, and $M_i\subseteq\Occ_{p_{r+1}}(p_i)$.
Occurrences inside higher blocks are preserved by the inheritance identity \eqref{eq:inheritance}.

\medskip 
\noindent\textbf{Claim 1:} $\Fr_{i,r+1}\subseteq S_i^6\Occ_{p_{r+1}}(p_i)$ for  $1\le i \le r$.

\begin{proof}[Proof of Claim 1]
Fix $1\le i\le r$ and $a\in\Fr_{i,r+1}$. 
Since $D_{r+1}\subseteq S_i^6V_{i+1}$ and $\Fr_{i,r+1}=D_{r+1}\setminus S_i^3V_{i+1}$, we have
\[
 a\in S_i^6V_{i+1}\setminus S_i^3V_{i+1}.
\]
Apply Lemma~\ref{lem:buffer} with $R=6$ to $S_i,D_i,C_i,L_i$ and $V_{i+1}$.
The buffer lemma produces $\ell\in L_i$ such that
\[
 D_i\ell\subseteq S_i^5V_{i+1}\setminus V_{i+1} =V_i\setminus V_{i+1},
 \qquad a\in S_i^4\ell.
\]
By \eqref{eq:wm_higher_core}, this domain avoids every higher block, so $\ell\in\Lambda_i$. 
Lemma~\ref{lem:packing} shows $\ell\in S_i^2t$ for some $t\in M_i$. 
Hence
\[
 a\in S_i^4S_i^2t=S_i^6t
   \subseteq S_i^6\Occ_{p_{r+1}}(p_i).
\]
This proves $\Fr_{i,r+1}\subseteq S_i^6\Occ_{p_{r+1}}(p_i)$.
\end{proof}

\medskip 
\noindent\textbf{Claim 2:} $D_{r+1}\subseteq S_i^8\Occ_{p_{r+1}}(p_i)$ for $1\le i \le r$.

\begin{proof}[Proof of Claim 2]
Fix $i\le r$ and $a\in D_{r+1}$. We distinguish four cases.

\emph{Case 1: $a$ belongs to a higher block.}
Write $a=bt\in D_kt\in\mathcal H_i$. A recurrence block has $k>i$, while each bridge block has $k=r$;
for the central block take $k=r$ and $t=e$, which also permits $k=i=r$. 
The induction hypothesis in Step~2 states that $D_k\subseteq S_i^8\Occ_{p_k}(p_i)$, so there exists $u\in\Occ_{p_k}(p_i)$ with $b\in S_i^8u$.
Since $D_kt$ carries an exact copy of $p_k$, the centre $ut$ is in $\Occ_{p_{r+1}}(p_i)$. 
Therefore
\[
 a\in S_i^8(ut).
\]

For the remaining cases, assume $a\notin\bigcup\mathcal H_i$.
Use $G=C_iL_i$ to write
\begin{equation}\label{eq:wm_candidate_centre}
 a=c\ell,\qquad c\in C_i\subseteq S_i,\quad\ell\in L_i.
\end{equation}
Every $y=d\ell\in D_i\ell$ satisfies
\begin{equation}\label{eq:wm_candidate_distance}
 ay^{-1}=cd^{-1}\in S_i^2,
 \qquad ya^{-1}=dc^{-1}\in S_i^2.
\end{equation}

\emph{Case 2: $D_i\ell$ is not contained in $V_i$.}
Choose $y\in D_i\ell\setminus V_i$. 
If $a\in S_i^3V_{i+1}$, then \eqref{eq:wm_candidate_distance} would show $y\in S_i^5V_{i+1}=V_i$, a contradiction. 
Consequently $a\in\Fr_{i,r+1}$. 
The frame covering already proved produces
\[
 a\in S_i^6\Occ_{p_{r+1}}(p_i)
   \subseteq S_i^8\Occ_{p_{r+1}}(p_i).
\]

For the two remaining cases, assume $D_i\ell\subseteq V_i$.

\emph{Case 3: $D_i\ell$ meets a higher block $D_kt$.}
Choose
\[
 y=bt\in D_i\ell\cap D_kt,\qquad b\in D_k.
\]
Since $a$ lies outside every higher block, $at^{-1}\notin D_k$.
On the other hand, \eqref{eq:wm_candidate_distance} shows
\begin{equation*}
 (at^{-1})b^{-1}=ay^{-1}\in S_i^2.
\end{equation*}
If $k>i$ and $b\notin\Fr_{i,k}$, then $b\in S_i^3V_{i+1}^{(k)}$ and $at^{-1}\in S_i^5V_{i+1}^{(k)}=V_i^{(k)}\subseteq D_k$, a contradiction. Thus $b\in\Fr_{i,k}$. 
Since $\Fr_{i,k}\subseteq S_i^6\Occ_{p_k}(p_i)$ by Step~2, choose
\[
 u\in\Occ_{p_k}(p_i),\qquad b\in S_i^6u.
\]
The inherited centre $ut$ belongs to $\Occ_{p_{r+1}}(p_i)$, and
\begin{equation*}
 a\in S_i^2y=S_i^2bt
   \subseteq S_i^2S_i^6ut=S_i^8ut.
\end{equation*}

The possibility $k=i$ occurs only when $i=r$ and the block is central or a bridge block. 
In that case $b\in D_i\subseteq S_i$ and \eqref{eq:wm_candidate_distance} produces $a\in S_i^2bt\subseteq S_i^3t$.
The centre $t\in\Occ_{p_{r+1}}(p_i)$ proves the desired covering. This case does not require a frame with equal indices.

\emph{Case 4: $D_i\ell$ is contained in $V_i$ and avoids every higher block.}
Then $\ell\in\Lambda_i$. 
By Lemma~\ref{lem:packing}, some $t\in M_i$ satisfies $\ell\in S_i^2t$. 
Using \eqref{eq:wm_candidate_centre} and the exact copy on $D_it$, we get
\[
 a\in S_i^3t\subseteq S_i^8\Occ_{p_{r+1}}(p_i).
\]

These cases exhaust $D_{r+1}$. 
Hence $D_{r+1}\subseteq S_i^8\Occ_{p_{r+1}}(p_i)$  holds for each $1\leq i\le r$.
\end{proof}

It remains to define the data needed for the next induction step.  
Since $p_{r+1}(e)=1$ and $\suppplus(p_{r+1})\subseteq F$, put
\[
 K_{r+1}:=\suppplus(p_{r+1}),
 \qquad L_{r+1}:=\Int_{K_{r+1}}(F).
\]
By the observation at the beginning of the proof, $L_{r+1}$ is right syndetic. 
Choose a finite set $C_{r+1}\ni e$ with $G=C_{r+1}L_{r+1}$.  
Choose a finite symmetric set $S_{r+1}\ni e$ such that
\[
D_{r+1}\cup C_{r+1}\subseteq S_{r+1}.
\]
Then $D_{r+1}\subseteq S_{r+1}^8\Occ_{p_{r+1}}(p_{r+1})$, because $D_{r+1}\subseteq S_{r+1}$ and
$e\in\Occ_{p_{r+1}}(p_{r+1})$.  
Thus all data required in the induction hypothesis are constructed at level $r+1$.
This completes the induction.

\medskip
\noindent\textbf{Step 4: Passage to the limit and verification of the conclusion}
Since $g_r\in B_r\subseteq D_r$ for every $r$, the increasing domains $D_r$ exhaust $G$. 
The compatibility relations $p_{r+1}|_{D_r}=p_r$ define a unique configuration $x\in2^G$ with $x|_{D_r}=p_r$.
Applying Lemma~\ref{lem:finite_covers} with the fixed finite sets $Q_i=S_i^8$  shows that $x$ is minimal, 
$x(e)=1$, and $\suppplus(x)\subseteq F$.
Put $X=\overline{Gx}$; then $X$ is minimal.

For each $k$, let $W_k=X\cap U[p_k]$. 
These clopen sets form a decreasing neighbourhood basis at $x$. 
Fix $r\ge k$ and $\gamma\in H_rc_rH_r$.
The two bridge copies installed at stage $r$ are preserved in $x$, so
$t_{r,\gamma}x\in W_k$ and $\gamma t_{r,\gamma}x\in W_k$.
Thus $\gamma\in N_G(W_k,W_k)$, and
\[
H_rc_rH_r\subseteq N_G(W_k,W_k).
\]
Given a finite $H\subseteq G$, choose $r\ge k$ with $H\subseteq H_r$.
It follows that $Hc_rH\subseteq N_G(W_k,W_k)$, so this self-hitting set is bitranslation thick.

Fix $m\ge2$ and non-empty open rectangles $\prod_{i=1}^mU_i$ and $\prod_{i=1}^mV_i$ in $X^m$.
Since $Gx$ is dense and the $W_k$ form a neighbourhood basis at $x$,
there are a common $k$ and elements $a_i,b_i\in G$ with $a_iW_k\subseteq U_i$ and $b_iW_k\subseteq V_i$ for every $i$.
Apply bitranslation thickness to
\[
H=\{b_i^{-1}:1\le i\le m\}\cup\{a_i:1\le i\le m\}.
\]
There is $c\in G$ such that $b_i^{-1}ca_i\in N_G(W_k,W_k)$ for every $i$.
Consequently $c(a_iW_k)\cap b_iW_k\ne\varnothing$, and hence $cU_i\cap V_i\ne\varnothing$ for every $i$.
Thus the diagonal action on $X^m$ is topologically transitive for every $m\ge2$, so $X$ is weakly mixing of all finite orders.
Finally, the identity cylinder $U_0=\{y\in X:y(e)=1\}$ is a clopen neighbourhood of $x$, and the shift formula identifies
\[
N_G(x,U_0)=\suppplus(x)\subseteq F.\qedhere
\]
\end{proof}

Now we are ready to prove Theorem~\ref{thm:wm_allorders}.
\begin{proof}[Proof of 
Theorem~\ref{thm:wm_allorders}]
Let $F\subseteq G$ with $e\in F$.

(1)$\implies$(2).
Let $A=N_G(x,U)\subseteq F$, where $X$ is minimal and weakly mixing of all finite orders. 
For finite $K\subseteq A$, put
\[
V_K=\bigcap_{k\in K}k^{-1}U.
\]
Then $V_K$ is a non-empty open neighbourhood of $x$, and
\[
\Int_K(A)=N_G(x,V_K).
\]
For $g\in G$,
\[
\Int_{K\cup Kg}(A)
=N_G(x,V_K\cap g^{-1}V_K).
\]
Minimality implies that this return set is right syndetic exactly when $V_K\cap g^{-1}V_K\neq\varnothing$. 
Consequently
\[
\Gamma_A(K)=N_G(V_K,V_K).
\]
By Lemma~\ref{lem:allorders_self_hitting}, this set is bitranslation thick. 
Lemma~\ref{lem:interior_vs_product} consequently shows
\[
\Int_K(A)\Int_K(A)^{-1}
=\Gamma_A(K),
\]
which proves (2).

(2)$\implies$(3). 
By Lemma~\ref{lem:interior_vs_product}, the hypotheses in (2) imply
\[
\Gamma_A(K)=\Int_K(A)\Int_K(A)^{-1}
\]
for every finite $K\subseteq A$. 
These bridge sets $\Gamma_A(K)$ are bitranslation thick, so Proposition~\ref{prop:allorders_realization} applies to $A$. 
The construction produces $x\in2^G$ with $x(e)=1$, $\suppplus(x)\subseteq A\subseteq F$, 
and $X=\overline{Gx}$ minimal and weakly mixing of all finite orders. 
Since $x=\one_{\suppplus(x)}$, condition (3) follows.

(3)$\implies$(1). 
In $X=\overline{G\one_A}$, let
\[
U=\{y\in X:y(e)=1\}.
\]
Then $U$ is clopen, $\one_A\in U$, and
\[
N_G(\one_A,U)=A\subseteq F.  \qedhere
\]
\end{proof}

Applying Theorem~\ref{thm:wm_allorders}, we obtain the following further equivalent characterizations of return-time sets of minimal $G$-flows that are weakly mixing of all finite orders.

\begin{proposition}\label{prop:more_equi_wm_sets}
Let $F\subseteq G$ with $e\in F$. The following four assertions are equivalent:
\begin{enumerate}[label=\textup{(\arabic*)}]
\item There exists a minimal $G$-flow $X$ that is topologically weakly mixing of all finite orders, a point $x\in X$, and a non-empty open set $U\ni x$ such that $N_G(x,U)\subseteq F$.
\item there exists a sequence of decreasing right syndetic sets
\[
F_1\supseteq F_2\supseteq\cdots,\qquad F_n\subseteq F,
\]
such that for every $n$ and every finite $K\subseteq F_n$, some $m\ge n$ satisfies $KF_m\subseteq F_n$, 
and $\{g\in G:F_n\cap g^{-1}F_n\text{ is right syndetic}\}$ is bitranslation thick for every $n$.
\item There exists a minimal $G$-flow $(X,G)$ that is weakly mixing of all finite orders, a point $x\in X$, 
and a non-empty open set $U\subseteq X$ with $x\in\overline U$ such that $N_G(x,U)\cup\{e\}\subseteq F$.
\item $F$ belongs to a syndetic idempotent filter $\mathscr{F}$ such that for every $B\in\mathscr{F}$,
\[
\{g\in G\colon B\cap g^{-1} B\text{ is right syndetic}\}  
\]
is bitranslation thick.
\end{enumerate}
\end{proposition}
\begin{proof}
(1)$\implies$ (2). 
By Theorem~\ref{thm:wm_allorders}(3), choose a binary point $x$ with minimal orbit closure $X_0$ weakly mixing of all finite orders and identity cylinder $U_1=\{z\in X_0:z(e)=1\}$ such that $N_G(x,U_1)\subseteq F$. Extend $U_1$ to a decreasing clopen neighbourhood basis $U_1\supseteq U_2\supseteq\cdots$ at $x$ in $X_0$. 
Put $F_n=N_G(x,U_n)$. 
Each $F_n$ is right syndetic, and shrinking the basis produces $KF_m\subseteq F_n$. 
Also
\[
\{g\in G:F_n\cap g^{-1}F_n\text{ is right syndetic}\}=N_G(U_n,U_n),
\]
that is bitranslation thick by Lemma~\ref{lem:allorders_self_hitting}.

(2)$\implies$(1). Put $A=F_1\cup\{e\}$. 
Given finite $K\subseteq A$, choose $m$ with $F_m\subseteq F_1$ and $(K\cap F_1)F_m\subseteq F_1$. 
Then $F_m\subseteq\Int_K(A)$. If $g\in\{h\in G:F_m\cap h^{-1}F_m\text{ is right syndetic}\}$, then
\[
F_m\cap g^{-1}F_m\subseteq\Int_{K\cup Kg}(A),
\]
so $g\in\Gamma_A(K)$. 
Hence $\Gamma_A(K)$ is bitranslation thick, and Theorem~\ref{thm:wm_allorders} proves (1).

(1)$\implies$(3). This implication is immediate.

(3)$\implies$(1). Put $F_0=N_G(x,U)\cup\{e\}$. 
For finite $K\subseteq F_0$, let
\[
H=K\setminus\{e\},\qquad Q=\bigcap_{h\in H}h^{-1}U,\qquad V=U\cap Q.
\]
Then $V$ is non-empty and open. 
If $g\in N_G(V,V)$, then $W=V\cap g^{-1}V$ is non-empty and
\[
N_G(x,W)\subseteq\Int_{K\cup Kg}(F_0).
\]
The left side is right syndetic, so $g\in\Gamma_{F_0}(K)$.
Therefore
\[
N_G(V,V)\subseteq\Gamma_{F_0}(K).
\]
Lemma~\ref{lem:allorders_self_hitting} shows that the left side is bitranslation thick, so Theorem~\ref{thm:wm_allorders} proves (1).

(1)$\implies$(4). Let
\[
\mathscr F_x=\{N_G(x,V):V\text{ open and }x\in V\}^{\uparrow}.
\]
This is a syndetic filter and it is idempotent: if $R=N_G(x,V)$ and $g\in R$, then
\[
g^{-1}R=N_G(x,g^{-1}V)\in\mathscr F_x,
\]
so $R\subseteq R-\mathscr F_x$. If $R=N_G(x,V)$, then
\[
\{g\in G:R\cap g^{-1}R\text{ is right syndetic}\}=N_G(V,V),
\]
that is bitranslation thick by Lemma~\ref{lem:allorders_self_hitting}. 
Every $B\in\mathscr F_x$ contains such an $R$, consequently $\{g\in G:B\cap g^{-1}B\text{ is right syndetic}\}$ is bitranslation thick. 
Since the witnessing return set belongs to $\mathscr F_x$, $F\in\mathscr F_x$.

(4)$\implies$(1). 
Let $\mathscr F$ be as in (4), with $F\in\mathscr F$. 
Put
\[
D=F-\mathscr F,\qquad A=F\cap D.
\]
Then $e\in A$. If $g\in A$, then $g\in D=F-\mathscr F$, so $g^{-1}F\in\mathscr F$. 
Idempotence implies
\[
g^{-1}D=(g^{-1}F)-\mathscr F\in\mathscr F,
\]
and consequently $g^{-1}A=g^{-1}F\cap g^{-1}D\in\mathscr F$. 
Consequently, for every finite $K\subseteq A$,
\[
\Int_K(A)=\bigcap_{k\in K}k^{-1}A\in\mathscr F.
\]
Thus $\Int_K(A)$ is right syndetic. 
Moreover,
\[
\Gamma_A(K)=\{g\in G:\Int_K(A)\cap g^{-1}\Int_K(A)\text{ is right syndetic}\}
\]
is bitranslation thick by the defining property of $\mathscr F$. 
Lemma~\ref{lem:interior_vs_product} now shows that
\[
\Int_K(A)\Int_K(A)^{-1}
\]
is bitranslation thick.  
\end{proof}

\subsection{Return-time sets of topologically strongly mixing minimal flows}\label{subsec:strong_mixing}
The preceding subsection identifies bitranslation thickness as the bridge condition for weak mixing of all finite orders.  We now replace that condition by cofiniteness and run the corresponding construction for topologically strongly mixing minimal flows.
 
\begin{lemma}\label{lem:many_legal_centres}
Let $L\subseteq G$ be right syndetic, let $D\subseteq G$ be finite and non-empty, and let $N\ge1$ be an integer.  
There is a finite set $R\subseteq G$, with $e\in R$, such that for every finite $W\subseteq G$, all
\[
b\in G\setminus R^{-1}D^{-1}W
\]
admit $N$ distinct elements $u_1,\ldots,u_N\in L\cap Rb$ satisfying $Du_j\cap W=\varnothing$ for every $j$.
\end{lemma}

\begin{proof}
Choose a finite $C\subseteq G$ with $G=CL$.  
Choose distinct $r_1,\ldots,r_N\in G$ so that the finite sets $C^{-1}r_j$ are pairwise disjoint, and put
\[
R=\{e\}\cup\bigcup_{j=1}^N C^{-1}r_j.
\]
For $b\in G$, write $r_jb=c_ju_j$ with $c_j\in C$ and $u_j\in L$.
Then $u_j=c_j^{-1}r_jb\in L\cap Rb$, and the $u_j$ are distinct.
If $Du_j\cap W\ne\varnothing$, then $u_j=d^{-1}w$ for some $d\in D,w\in W$, so
$b=(c_j^{-1}r_j)^{-1}d^{-1}w\in R^{-1}D^{-1}W$.
Thus, outside the displayed finite exceptional set, all $u_j$ are legal.
\end{proof}

\begin{proposition}\label{prop:strong_mixing_realization}
Assume $e\in A\subseteq G$ and $\Gamma_A(K)$ is cofinite for every finite $K\subseteq A$.  
Then there is $x\in2^G$ with
\[
x(e)=1,\qquad \suppplus(x)\subseteq A,
\]
such that $X=\overline{Gx}$ is minimal and topologically strongly mixing.
\end{proposition}

\begin{proof}
Enumerate $G=\{g_1,g_2,\ldots\}$, with $g_1=e$, 
and enumerate all pairs $(k_r,\gamma_r)\in\mathbb N\times G$, with $k_r\le r$, so that every pair occurs. 
For finite $K\subseteq A$, write
\[
L_K:=\Int_K(A).
\]
The set $L_K$ is right syndetic: choose $g\in\Gamma_A(K)$; then $\Int_{K\cup Kg}(A)\subseteq L_K$ is right syndetic.

We construct by induction finite domains $D_i$, compatible binary patterns
\[
p_i\in\{0,1\}^{D_i},
\]
and the auxiliary data needed for the mixing, separation, and covering arguments. At each stage, the construction preserves the previously defined binary pattern, inserts the required mixing bridge blocks together with the corresponding guard blocks, and maintains the frame and global covering invariants used in the passage to the limit.

\medskip 
\noindent\textbf{Step 1: Initial construction.} 
Put $P_0=B_1=D_1=\{e\}$, $p_1(e)=1$, $K_1=\suppplus(p_1)$, and $L_1=L_{K_1}$. 
Choose finite $C_1^0\ni e$ with $G=C_1^0L_1$. 
Put
\[
M_1=2|D_1|,\qquad N_1=|D_1|M_1+1.
\]
Choose $R_1$ from Lemma~\ref{lem:many_legal_centres} for $L_1,D_1,N_1$. 
Enlarge $C_1^0$ to a finite $C_1\ni e$ containing $C_1^0\cup R_1\cup P_0$, and choose $S_1=S_1^{-1}$ with $D_1\cup C_1\subseteq S_1$.

\medskip
\noindent\textbf{Step 2: Induction hypothesis.}
Assume $B_i,D_i,p_i,C_i,R_i,S_i$ have been constructed for $1\le i\le r$, with $p_{i+1}|_{D_i}=p_i$, $p_i(e)=1$, and $\suppplus(p_i)\subseteq A$. 
Put
\[
P_{i-1}=S_1^5\cdots S_{i-1}^5,\quad D_i=P_{i-1}B_i,\quad
K_i=\suppplus(p_i),\quad  L_i=L_{K_i}.
\]
At level $i$, first choose a finite $C_i^0\ni e$ with $G=C_i^0L_i$. 
Put
\[
M_i=2\sum_{m\le i}|D_m|,\qquad N_i=|D_i|M_i+1,
\]
and choose $R_i$ from Lemma~\ref{lem:many_legal_centres} for $L_i,D_i,N_i$. 
Enlarge $C_i^0$ to an increasing finite $C_i\ni e$ containing
\[
C_i^0\cup R_i\cup P_{i-1}\cup C_1\cup\cdots\cup C_{i-1}.
\]
Then $G=C_iL_i$, $R_i\subseteq C_i$, and $P_{i-1}\subseteq C_i$. 
Choose a finite symmetric $S_i=S_i^{-1}$ with
\[
D_i\cup C_i\subseteq S_i.
\]
For $i<j$, define
\[
V_i^{(j)}=S_i^5\cdots S_{j-1}^5B_j,\qquad
V_j^{(j)}=B_j,\qquad
\mathsf{Fr}_{i,j}=D_j\setminus S_i^3V_{i+1}^{(j)}.
\]
Maintain the frame and global covering invariants in \eqref{eq:strong_mixing_covering_invariants}:
\begin{equation}
\mathsf{Fr}_{i,j}\subseteq S_i^6\Occ_{p_j}(p_i),\qquad
D_j\subseteq S_i^8\Occ_{p_j}(p_i)\label{eq:strong_mixing_covering_invariants}
\end{equation}
for $i<j$, together with $D_i\subseteq S_i^8\Occ_{p_i}(p_i)$. 
Since $e\in B_i$, $P_{i-1}\subseteq D_i$, consequently $P_{i-1}\subseteq C_i\subseteq S_i$. 
Thus
\[
V_i^{(j)}\subseteq D_j\subseteq S_i^6V_{i+1}^{(j)}.
\]
Consequently,
\begin{equation}
b\in D_j,\quad a\notin D_j,\quad a\in S_i^2b
\quad\Longrightarrow\quad b\in\mathsf{Fr}_{i,j}.\label{eq:strong_mixing_frame_boundary}
\end{equation}

\medskip
\noindent\textbf{Step 3: Inductive extension.}
Let $(k,\gamma)=(k_r,\gamma_r)$ and $Q_k=D_k^{-1}D_k$. 
If $\gamma\notin\Gamma_A(K_k)\setminus Q_k$, set $\mathcal J_r=\varnothing$. 
Otherwise let
\[
L_{k,\gamma}=\Int_{K_k\cup K_k\gamma}(A).
\]
Put $W_0=D_r$ and, for $k<i\le r$, define
\[
E_i(W_0)=C_i^{-1}D_i^{-1}W_0.
\]
Choose $t\in L_{k,\gamma}$ outside the finite set
\[
\bigcup_{i=k+1}^{r}\bigl(E_i(W_0)\cup\gamma^{-1}E_i(W_0)\bigr)
\cup D_k^{-1}D_r\cup\gamma^{-1}D_k^{-1}D_r.
\]
Then $D_kt$ and $D_k\gamma t$ avoid $D_r$ and are disjoint. 
Put
\[
u_k^{(0)}=t,\qquad u_k^{(1)}=\gamma t.
\]
For $i=k+1,\ldots,r$, apply Lemma~\ref{lem:many_legal_centres} with $W=W_0$ to the bases $t$ and $\gamma t$. 
Since $R_i\subseteq C_i$, the exceptional sets $E_i(W_0)$ suffice. 
We obtain $N_i$ legal candidates on each side. 
Select $u_i^{(0)},u_i^{(1)}$ successively so that all $D_i u_i^{(\nu)}$ are pairwise disjoint and avoid $D_r$ and the two core domains. 
At most $|D_i|M_i=N_i-1$ candidates are forbidden at each choice. 
Put exact copies of $p_k$ on the two core domains and exact copies of $p_i$ on the guard domains; call the family $\mathcal J_r$.

For $m<i$, write $u_m^{(\nu)}=\alpha_m b$ and $u_i^{(\nu)}=\alpha_i b$, where $b=t$ or $b=\gamma t$, $\alpha_m\in C_m$, 
and $\alpha_i\in C_i$; at the core level take $\alpha_k=e$. 
Since $D_m\subseteq S_m\subseteq P_{i-1}\subseteq C_i$ for $m<i$, and $C_m,C_i\subseteq S_i$, then
\begin{equation}
D_m u_m^{(\nu)}(u_i^{(\nu)})^{-1}
\subseteq D_mC_mC_i^{-1}\subseteq S_i^3.\label{eq:strong_mixing_guard_distance}
\end{equation}

Set
\[
P_r=S_1^5\cdots S_r^5,\quad
B_{r+1}=D_r\cup\{g_1,\ldots,g_{r+1}\}\cup\bigcup_{J\in\mathcal J_r}J,\quad
D_{r+1}=P_rB_{r+1}.
\]
Write $V_i=V_i^{(r+1)}$. 
The product geometry is recorded in \eqref{eq:strong_mixing_product_geometry}:
\begin{equation}
V_i=S_i^5V_{i+1},\qquad
V_i\subseteq D_{r+1}\subseteq S_i^6V_{i+1}\quad(1\le i\le r).\label{eq:strong_mixing_product_geometry}
\end{equation}

Keep the central, bridge, and guard assignments fixed. 
Process $i=r,r-1,\ldots,1$. 
Let $\mathcal H_i$ consist of these fixed blocks and the recurrence blocks already placed at levels $>i$, and define
\[
\Lambda_i=\{\ell\in L_i:D_i\ell\subseteq V_i,\ 
D_i\ell\cap H=\varnothing\text{ for every }H\in\mathcal H_i\}.
\]
Choose inclusion-maximal $\mathcal M_i\subseteq\Lambda_i$ with pairwise disjoint $D_it$, place exact copies of $p_i$ on these domains, 
and put $0$ on all remaining sites. Since every fixed central, bridge, guard, or already placed higher-level block is contained in $V_{i+1}^{(r+1)}$, we have
\[
\bigcup\mathcal H_i\subseteq V_{i+1}^{(r+1)}.
\]
Thus a buffer domain in $V_i\setminus V_{i+1}$ avoids every member of $\mathcal H_i$. 
This produces $p_{r+1}\supseteq p_r$, with positive support in $A$, and $\mathcal M_i\subseteq\Occ_{p_{r+1}}(p_i)$. 
Occurrences inherited from higher blocks are preserved by
\[
u\in\Occ_{p_j}(p_i)\Longrightarrow
uv\in\Occ_{p_{r+1}}(p_i)
\]
when $p_j$ is copied on $D_jv$.

\medskip
\noindent\textbf{Claim 1.}
For $1\le i\le r$,
\[
\mathsf{Fr}_{i,r+1}\subseteq S_i^6\Occ_{p_{r+1}}(p_i).
\]
\begin{proof}[Proof of Claim 1]
If $a\in\mathsf{Fr}_{i,r+1}$, Lemma~\ref{lem:buffer} produces $\ell\in L_i$ with
\[
D_i\ell\subseteq V_i\setminus V_{i+1},\qquad a\in S_i^4\ell.
\]
This domain avoids $\mathcal H_i$, so $\ell\in\Lambda_i$. 
Lemma~\ref{lem:packing} produces $\ell\in S_i^2t$ for some $t\in \mathcal M_i$, and consequently $a\in S_i^6t$.
\end{proof}

\medskip
\noindent\textbf{Claim 2.}
For $1\le i\le r$,
\[
D_{r+1}\subseteq S_i^8\Occ_{p_{r+1}}(p_i).
\]

\begin{proof}[Proof of Claim 2]
First consider a point $a$ in one of the fixed central, bridge, or guard blocks, or in a recurrence block $D_jv$ already placed at a level $j>i$.
 Classify that block by the pattern level $q$ that it actually carries. 
A recurrence block $D_jv$ carries $p_j$, each bridge core block $D_ku_k^{(\nu)}$ carries $p_k$, and each guard block $D_mu_m^{(\nu)}$ carries $p_m$. 
If $q>i$, the old global invariant for $p_q$, followed by inheritance, covers the block. 
If $q=i$, its own centre is an occurrence centre for $p_i$. 
If $q<i$, the block is a fixed lower guard or bridge block, and \eqref{eq:strong_mixing_guard_distance} places it within $S_i^3u_i^{(\nu)}$, where $u_i^{(\nu)}$ is an occurrence centre for $p_i$. 
Lower recurrence blocks are not fixed at this stage and will be handled by the maximal packing below.

Now let $a$ lie outside the blocks just considered. Write $a=c\ell$, $c\in C_i\subseteq S_i$, $\ell\in L_i$. 
For $y=d\ell$, $d\in D_i$,
\begin{equation}
ay^{-1}=cd^{-1}\in S_i^2,\qquad ya^{-1}=dc^{-1}\in S_i^2.\label{eq:strong_mixing_candidate_distance}
\end{equation}
If $D_i\ell\not\subseteq V_i$, choose $y\in D_i\ell\setminus V_i$. 
If $a\in S_i^3V_{i+1}$, then \eqref{eq:strong_mixing_candidate_distance} would show $y\in S_i^5V_{i+1}=V_i$, a contradiction. 
Consequently, $a\in\mathsf{Fr}_{i,r+1}$, and Claim 1 applies. 
Suppose $D_i\ell\subseteq V_i$. 
If it meets a block $D_jv$ carrying $p_j$ with $j>i$, write $y=bv$ in the intersection, where $b\in D_j$. 
Since $a$ is outside that block, $av^{-1}\notin D_j$; \eqref{eq:strong_mixing_candidate_distance} and \eqref{eq:strong_mixing_frame_boundary} show $b\in\mathsf{Fr}_{i,j}$. 
The old frame invariant produces $u\in\Occ_{p_j}(p_i)$ with $b\in S_i^6u$, and inheritance produces $uv\in\Occ_{p_{r+1}}(p_i)$, so $a\in S_i^8uv$. 
If it meets a fixed level-$i$ central, bridge, or guard block $D_iu$, choose $y=bu$ in the intersection with $b\in D_i\subseteq S_i$. Equation~\eqref{eq:strong_mixing_candidate_distance} implies $a\in S_i^2y\subseteq S_i^3u$, and $u$ is an occurrence centre for $p_i$. This includes the central block $D_r=D_i$ when $i=r$. 
If it meets a fixed lower bridge or guard block $D_mu_m^{(\nu)}$, then \eqref{eq:strong_mixing_guard_distance} and \eqref{eq:strong_mixing_candidate_distance} show $a\in S_i^5u_i^{(\nu)}$. 
In every remaining case, $D_i\ell$ avoids the fixed blocks in $\mathcal H_i$ and the level-$i$ blocks, so $\ell\in\Lambda_i$; this includes the possibility that it meets a lower recurrence block. 
Maximality and Lemma~\ref{lem:packing} show $a\in S_i^3t$ for some $t\in \mathcal M_i$. These cases prove the claim.
\end{proof}

For the next level, set $K_{r+1}=\suppplus(p_{r+1})$, $L_{r+1}=L_{K_{r+1}}$, 
and choose $C_{r+1}^0\ni e$ with $G=C_{r+1}^0L_{r+1}$. 
Put
\[
M_{r+1}=2\sum_{m\le r+1}|D_m|,\qquad
N_{r+1}=|D_{r+1}|M_{r+1}+1,
\]
choose $R_{r+1}$ by Lemma~\ref{lem:many_legal_centres} for $L_{r+1},D_{r+1},N_{r+1}$, enlarge $C_{r+1}^0$ to a finite increasing $C_{r+1}$ containing
\[
C_{r+1}^0\cup R_{r+1}\cup P_r\cup C_1\cup\cdots\cup C_r,
\]
and choose $S_{r+1}=S_{r+1}^{-1}$ with
\[
D_{r+1}\cup C_{r+1}\subseteq S_{r+1}.
\]
The equal-index invariant follows from $D_{r+1}\subseteq S_{r+1}$ and $e\in\Occ_{p_{r+1}}(p_{r+1})$. 
Since $g_{r+1}\in B_{r+1}\subseteq D_{r+1}$, the domains are increasing and exhaust $G$. 
The new bridge and guard domains avoid $D_r$ and are contained in $D_{r+1}$; all later patterns extend $p_{r+1}$, 
so these exact copies are never overwritten. 
This completes the induction.

\medskip
\noindent\textbf{Step 4: Passage to the limit and  verification of the conclusion.}
Since $g_r\in B_r\subseteq D_r$, the increasing domains $D_r$ exhaust $G$. 
Compatibility defines $x\in2^G$. 
By Lemma~\ref{lem:finite_covers}, with $Q_i=S_i^8$, minimality, $x(e)=1$, and $\suppplus(x)\subseteq A$ follow.

Put $X=\overline{Gx}$. For each $k$, let $W_k:=X\cap U[p_k]$.
These sets form a neighbourhood basis at $x$. 
Every $\gamma\in\Gamma_A(K_k)\setminus Q_k$ is processed at some stage, and the two core copies show $\gamma\in N_G(W_k,W_k)$. 
Since the left-hand set is cofinite, $N_G(W_k,W_k)$ is cofinite. 
For non-empty open $U,V\subseteq X$, choose $a,b\in G$ and $k$ with $aW_k\subseteq U$ and $bW_k\subseteq V$. Then
\[
N_G(aW_k,bW_k)=bN_G(W_k,W_k)a^{-1}
\]
is cofinite and contained in $N_G(U,V)$.
 Consequently, $X$ is topologically strongly mixing. 
 Finally,
\[
N_G\bigl(x,\{y\in X:y(e)=1\}\bigr)
=\suppplus(x)\subseteq A.\qedhere
\]
\end{proof}

Now we are ready to prove Theorem~\ref{thm:sm_sets}.
\begin{proof}[Proof of Theorem~\ref{thm:sm_sets}]
Let $F\subseteq G$ with $e\in F$.

(1)$\implies$(2). 
Let $A=N_G(x,U)\subseteq F$, where $X$ is minimal and topologically strongly mixing. 
For finite $K\subseteq A$, put $V_K=\bigcap_{k\in K}k^{-1}U$. 
Then
\[
\Int_K(A)=N_G(x,V_K)
\]
is right syndetic, and
\[
\Int_K(A)\Int_K(A)^{-1}
=N_G(V_K,V_K)
\]
is cofinite by strong mixing.

(2)$\implies$(3). 
By Lemma~\ref{lem:interior_vs_product}, $\Gamma_A(K)=\Int_K(A)\Int_K(A)^{-1}$ is cofinite for every finite $K\subseteq A$. Apply Proposition~\ref{prop:strong_mixing_realization} to $A$. 
The construction produces $x\in2^G$ with $x(e)=1$, positive support contained in $A$, and minimal, topologically strongly mixing orbit closure. 
Thus the positive support $C$ satisfies $e\in C\subseteq F$ and $x=\one_C$.

(3)$\implies$(1). 
In $X=\overline{G\one_A}$, the identity cylinder $U=\{y:y(e)=1\}$ is clopen and, by \eqref{eq:shift},
\[
N_G(\one_A,U)=A\subseteq F.\qedhere
\]
\end{proof}

Applying Theorem~\ref{thm:sm_sets}, we obtain the following further equivalent characterizations for return-time sets of topologically strongly mixing minimal $G$-flows.

\begin{proposition}\label{prop:more_equi_sm_sets}
Let $F\subseteq G$ with $e\in F$. The following four assertions are equivalent:
\begin{enumerate}[label=\textup{(\arabic*)}]
\item There exists a topologically strongly mixing minimal $G$-flow $X$, a point $x\in X$, 
and a non-empty open set $U\ni x$ such that $N_G(x,U)\subseteq F$.

\item There exists a sequence of decreasing right syndetic sets
\[
F_1\supseteq F_2\supseteq\cdots,\qquad F_n\subseteq F,
\]
such that for every $n$ and every finite $K\subseteq F_n$, some $m\ge n$ satisfies $KF_m\subseteq F_n$, 
and $\{g\in G:F_n\cap g^{-1}F_n\text{ is right syndetic}\}$ is cofinite for every $n$.
\item There exists a topologically strongly mixing minimal $G$-flow $(X,G)$, a point $x\in X$, and a non-empty
open set $U\subseteq X$ with $x\in\overline U$ such that
\[
N_G(x,U)\cup\{e\}\subseteq F.
\]
\item $F$ belongs to a syndetic, idempotent filter $\mathscr{F}$  on $G$ such that for every $B\in \mathscr{F}$
\[
\{g\in G\colon B\cap g^{-1} B\text{ is right syndetic}\}  
\]
is cofinite.
\end{enumerate}
\end{proposition}

\begin{proof}
(1)$\implies$(2). 
Choose a binary point $x$ with minimal, topologically strongly mixing orbit closure and $N_G(x,U_1)\subseteq F$, where $U_1$ is the identity cylinder. 
Choose a decreasing clopen neighbourhood basis $U_1\supseteq U_2\supseteq\cdots$ at $x$, and put $F_n=N_G(x,U_n)$. 
Each $F_n$ is right syndetic. For $g\in G$, the identity
\[
F_n\cap g^{-1}F_n=N_G(x,U_n\cap g^{-1}U_n)
\]
and minimality show that this intersection is right syndetic exactly when $U_n\cap g^{-1}U_n\ne\varnothing$. Hence
\[
\{g\in G:F_n\cap g^{-1}F_n\text{ is right syndetic}\}=N_G(U_n,U_n),
\]
which is cofinite by strong mixing. 
Shrinking the basis produces $KF_m\subseteq F_n$ for finite $K\subseteq F_n$.

(2)$\implies$(1). 
Put $A=F_1\cup\{e\}$. 
If $K\subseteq A$ is finite, choose $m$ with $KF_m\subseteq A$. 
Then $F_m\subseteq\Int_K(A)$, so this finite interior is right syndetic. If $F_m\cap g^{-1}F_m$ is right syndetic, then
\[
F_m\cap g^{-1}F_m\subseteq\Int_K(A)\cap g^{-1}\Int_K(A)=\Int_{K\cup Kg}(A),
\]
and therefore $g\in\Gamma_A(K)$. The cofinite shift set in (2) is consequently contained in $\Gamma_A(K)$. Lemma~\ref{lem:interior_vs_product} converts this inclusion into the cofinite difference-set condition in Theorem~\ref{thm:sm_sets}(2), which proves (1). 

(1)$\implies$(3). This implication is immediate.
 
(3)$\implies$(1).
Put $F_0=N_G(x,U)\cup\{e\}$.  
For a finite $K\subseteq F_0$, set $H=K\setminus\{e\}$, $Q=\bigcap_{h\in H}h^{-1}U$, and $V=U\cap Q$.  
Since $x\in\overline U$, the open set $V$ is non-empty.
Moreover,
\[
N_G(x,V)\subseteq\Int_K(F_0).
\]
The set $N_G(x,V)$ is right syndetic and $N_G(V,V)=N_G(x,V)N_G(x,V)^{-1}$ is cofinite by strong mixing. Hence the finite-interior condition and the cofinite bridge condition hold. Theorem~\ref{thm:sm_sets} then proves (1).

(1)$\implies$(4). Let
\[
\mathscr F_x=\{N_G(x,V):V\text{ open and }x\in V\}^{\uparrow}.
\]
Minimality makes every basic return set right syndetic, so $\mathscr F_x$ is a syndetic filter. 
It is idempotent: if $R=N_G(x,V)$ and $g\in R$, then
\[
g^{-1}R=N_G(x,g^{-1}V)\in\mathscr F_x,
\]
so $R\subseteq R-\mathscr F_x$, and upward closure produces $R-\mathscr F_x\in\mathscr F_x$. 
If $R=N_G(x,V)$, then
\[
\{g\in G:R\cap g^{-1}R\text{ is right syndetic}\}=N_G(V,V),
\]
that is cofinite by strong mixing. 
Every member of $\mathscr F_x$ contains such an $R$; its bridge set contains the cofinite bridge set for $R$, so it is cofinite. 
Since the witnessing return set belongs to $\mathscr F_x$, $F\in\mathscr F_x$.

(4)$\implies$(1). 
Let $\mathscr F$ be as in (4), with $F\in\mathscr F$, and put
\[
D=F-\mathscr F,\qquad A=F\cap D.
\]
Then $e\in A$. 
If $g\in A$, then $g^{-1}F\in\mathscr F$ and idempotence shows
\[
g^{-1}D=(g^{-1}F)-\mathscr F\in\mathscr F,
\]
so $g^{-1}A\in\mathscr F$. 
Consequently, for every finite $K\subseteq A$,
\[
\Int_K(A)=\bigcap_{k\in K}k^{-1}A\in\mathscr F
\]
is right syndetic. 
Its bridge set
\[
\Gamma_A(K)=\{g\in G:\Int_K(A)\cap g^{-1}\Int_K(A)\text{ is right syndetic}\}
\]
is cofinite by the defining property of $\mathscr F$. 
Lemma~\ref{lem:interior_vs_product} now shows that
\[
\Int_K(A)\Int_K(A)^{-1}
\]
is cofinite.  
\end{proof}

\subsection{A free strongly mixing partner disjoint from a prescribed minimal flow}
\label{subsec:free_strong_mixing_disjoint}
We now use separated covering to synchronize prescribed finite patterns with arbitrary open sets in the given minimal flow; this synchronization will force the product with the constructed flow to be minimal.
For a finite set $D\subseteq G$, we say that $E\subseteq G$ is \emph{$D$-separated}
if $Ds\cap Ds^{\prime}=\varnothing$ whenever $s\ne s^{\prime}$ belong to $E$. 
We use the \emph{separated covering property} (SCP): for every finite $D\subseteq G$ and every non-empty open $U\subseteq X$, 
there is a finite $D$-separated set $E\subseteq G$ such that $E^{-1}U=X$. 

Every compact metrizable minimal $G$-flow has the following property, proved by Glasner et al.~\cite{GTWZ21}
\begin{theorem}[{\cite[Theorem~6.2]{GTWZ21}}]
Every compact metrizable minimal $G$-flow has the separated covering property.
\end{theorem} 

\begin{lemma}\label{lem:relative_sync}
Let $X$ be a compact metrizable minimal $G$-flow, 
and let $(U_m)_{m\ge1}$ be a countable basis of non-empty open subsets of $X$.
Let $Y$ be a minimal $G$-flow, let $y_0\in Y$, and let
\[
W_1\supseteq W_2\supseteq\cdots
\]
be a decreasing neighbourhood basis at $y_0$.
Suppose that for every $m,k\ge1$ there exists a finite set $E_{m,k}\subseteq G$ and a point $y_{m,k}\in Y$ such that
\begin{equation}\label{eq:relative_sync_cover}
E_{m,k}^{-1}U_m=X,
\end{equation}
and
\begin{equation}\label{eq:relative_sync_pattern}
E_{m,k}y_{m,k}\subseteq W_k.
\end{equation}
Then $X\times Y$ is minimal. 
Consequently,
\[
X\perp Y.
\]
\end{lemma}

\begin{proof}
Fix $(x,y)\in X\times Y$ and a non-empty open rectangle $U\times V$. 
By minimality of $Y$, choose $a\in G$ and $k$ such that $aW_k\subseteq V$. Choose $m$ with $U_m\subseteq a^{-1}U$. 
Put $E=E_{m,k}$ and $y'=y_{m,k}$. 
The set $C=\bigcap_{s\in E}s^{-1}W_k$ is a non-empty open set containing $y'$ by \eqref{eq:relative_sync_pattern}. 
By minimality choose $h\in G$ with $hy\in C$. 
Equation~\eqref{eq:relative_sync_cover}, applied to $hx$, produces $s\in E$ with $shx\in U_m$. 
For this $s$, $shy\in W_k$. 
Consequently, $ashx\in aU_m\subseteq U$ and $ashy\in aW_k\subseteq V$. 
Thus every orbit in $X\times Y$ meets every non-empty open rectangle, so the product is minimal. 
Since both factors are minimal, every minimal subflow of $X\times Y$ projects onto both factors and hence is a joining. Thus disjointness forces the product to be minimal; conversely, a minimal product has no proper joining.
\end{proof}

\begin{lemma}
\label{lem:finite_avoidance}
Let $D\subseteq G$ be finite and non-empty, and let $M\ge0$ be an integer.
There is a finite set $R\subseteq G$, with $e\in R$, such that for every $b\in G$ and every finite $W\subseteq G$ with $|W|\le M$ there is $u\in Rb$ satisfying
\begin{equation*}
Du\cap W=\varnothing.
\end{equation*}
\end{lemma}

\begin{proof}
Choose $r_1,\ldots,r_{M+1}$ recursively so that the finite sets $Dr_j$ are pairwise disjoint. 
This is possible because, at the $j$th step, only the finite set $\bigcup_{\nu<j}D^{-1}Dr_\nu$ is forbidden. 
Put $R=\{e,r_1,\ldots,r_{M+1}\}$. The sets $Dr_jb$ remain pairwise disjoint. 
If every one met $W$, their pairwise disjointness would force $|W|\ge M+1$, a contradiction.
\end{proof}

\begin{lemma}\label{lem:three_color_extension}
Let $g\in G\setminus\{e\}$ and let $D\subseteq G$ be finite. 
Define the graph $\mathcal G(g,D)=(D,E)$ by
\[
E=\bigl\{\{h,gh\}:h\in D,\ gh\in D\bigr\}.
\]
Then every vertex of $\mathcal G(g,D)$ has degree at most two. 
If $P\subseteq D$ and $c:P\to\{0,1,2\}$ satisfies $c(h)\ne c(gh)$ whenever $h,gh\in P$, 
then there exists $\widetilde c:D\to\{0,1,2\}$ such that $\widetilde c|_P=c$ and
$\widetilde c(h)\ne\widetilde c(gh)$ whenever $h,gh\in D$.
\end{lemma}

\begin{proof}
For each $h\in D$, the only possible neighbors of $h$ in $\mathcal G(g,D)$ are $gh$ and $g^{-1}h$, so its degree is at most two. 
Order the vertices of  $D\setminus P$ as $v_1,\ldots,v_s$. 
Extend $c$ recursively. 
At the step $v_j$, at most two colors are forbidden by already colored neighbors of $v_j$; 
consequently one of the three colors is available. 
The resulting coloring
$\widetilde c$ extends $c$ and is proper on every edge of $\mathcal G(g,D)$.
\end{proof}

We are now ready to prove Theorem~\ref{thm:free_strong_mixing_disjoint}.
\begin{proof}[Proof of Theorem~\ref{thm:free_strong_mixing_disjoint}]
Fix a countable basis $(U_m)_{m\ge1}$ of non-empty open subsets of $X$. 
Write $G=\{g_1,g_2,\ldots\}$ with $g_1=e$. 
Enumerate $\mathbb N\times G$ as $(k_r,\gamma_r)_{r\ge1}$ and $\mathbb N^2$ as
$(\ell_r,m_r)_{r\ge1}$, with $k_r,\ell_r\le r$ and with every pair occurring. 
Let $\mathcal Z=\prod_{n\ge1}\{0,1,2\}^G$ with the shift action
$(gz)_n(h)=z_n(hg)$.

Before the first extension stage, we record the finite pattern data and the invariants that will be preserved. The induction tracks domain growth, color constraints enforcing freeness, mixing and synchronization blocks, and the frame and global covering estimates used in the limit.

\medskip
\noindent\textbf{Step 1: Initial construction.}
Put $P_0=B_1=D_1=\{e\}$ and choose a map
$c_{1,1}\colon D_1\to\{0,1,2\}$. Define
\[
q_1\colon D_1\to\{0,1,2\},
\qquad q_1(d)=c_{1,1}(d).
\]
Set
\[
T_1=\{e,g_2,g_2^{-1}\},\qquad
\widehat D_1=T_1D_1,\qquad \mu_1=|\widehat D_1|.
\]
Choose $R_1$ from Lemma~\ref{lem:finite_avoidance} with
$D=\widehat D_1$ and $M=\mu_1$.
Choose a finite symmetric set $S_1$ containing
$e,D_1,\widehat D_1,P_0$, and $R_1$, and put $P_1=S_1^5$.

\medskip
\noindent\textbf{Step 2: Induction hypothesis.}
Assume that the finite sets $B_i,D_i,R_i,S_i$ and the finite color patterns
\[
q_i\colon D_i\to\{0,1,2\}^{i},
\qquad q_i(d)=\bigl(c_{1,i}(d),\ldots,c_{i,i}(d)\bigr),
\]
have been constructed for $1\le i\le r$, where
$c_{n,i}\colon D_i\to\{0,1,2\}$ for $1\le n\le i$.
The domains satisfy $D_i=P_{i-1}B_i$ and, for $2\le i\le r$,
$D_{i-1}\cup\{g_1,\ldots,g_i\}\subseteq B_i$.
The patterns are compatible in the following sense: for $1\le i<r$ and every $d\in D_i$,
\[
q_{i+1}(d)|_{\{1,\ldots,i\}}=q_i(d),
\qquad\text{equivalently}\qquad
c_{n,i+1}(d)=c_{n,i}(d)\quad(1\le n\le i).
\]
They also satisfy the color invariant
\[
c_{n,i}(g_{n+1}h)\ne c_{n,i}(h)
\quad\text{whenever }1\le n\le i\le r,
\quad h,g_{n+1}h\in D_i.
\]
For each $1\le i\le r$, put
\[
T_i=\{e,g_2,g_2^{-1},\ldots,g_{i+1},g_{i+1}^{-1}\},
\qquad \widehat D_i=T_iD_i,
\qquad \mu_i=2\sum_{m<i}|\widehat D_m|+|\widehat D_i|.
\]
We call $\widehat D_i u=T_iD_i u$ the \emph{halo} of the domain $D_i u$.
The set $R_i$ has been chosen as in Lemma~\ref{lem:finite_avoidance}
with $D=\widehat D_i$ and $M=\mu_i$.
The finite symmetric set $S_i=S_i^{-1}$ contains $e$, $D_i$,
$\widehat D_i$, $P_{i-1}$, $R_i$, and all earlier finite control sets,
and
\[
P_i=S_1^5S_2^5\cdots S_i^5.
\]

For $1\le i\le j\le r$, define the occurrence set by
\[
\Occ_{q_j}(q_i)=\{t\in G:D_it\subseteq D_j,
\ c_{n,j}(dt)=c_{n,i}(d)\text{ for all }d\in D_i, 1\le n\le i\}.
\]
For $z=(c_n)_{n\ge1}\in\mathcal Z$, use the notation
\[
\Occ_z(q_i)=\{t\in G:c_n(dt)=c_{n,i}(d)\text{ for all }d\in D_i, 1\le n\le i\}.
\]
Whenever the data at level $j$ are fixed, put
\[
V_i^{(j)}=S_i^5S_{i+1}^5\cdots S_{j-1}^5B_j
\quad(1\le i<j),\qquad V_j^{(j)}=B_j,
\]
and define the level-$i$ frame of $D_j$ by
\[
\Fr_{i,j}=D_j\setminus S_i^3V_{i+1}^{(j)}\qquad(1\le i<j).
\]
Assume the frame covering invariant
\[
\Fr_{i,j}\subseteq S_i^6\Occ_{q_j}(q_i)
\qquad(1\le i<j\le r),
\]
and the global covering invariant
\[
D_j\subseteq S_i^8\Occ_{q_j}(q_i)
\qquad(1\le i\le j\le r).
\]
At level $1$ the color condition is vacuous,
$D_1\subseteq S_1^8\Occ_{q_1}(q_1)$ holds because
$e\in\Occ_{q_1}(q_1)$, and no frame is needed at equal indices.

\medskip
\noindent\textbf{Step 3: Inductive extension.}
At stage $r$, put $(k,\gamma)=(k_r,\gamma_r)$ and $Q_k=\widehat D_k^{-1}\widehat D_k$. 
If $\gamma\in Q_k$, set $\mathcal J_r^{\mathrm{mix}}=\varnothing$. 
Otherwise put
\begin{equation*}
\widehat H_r=\bigcup_{i=1}^r\widehat D_iR_i
\end{equation*}
and choose
\begin{equation*}
t\notin\widehat H_r^{-1}D_r\cup\gamma^{-1}\widehat H_r^{-1}D_r.
\end{equation*}
The exceptional set is finite. 
Consequently, every set $\widehat D_iR_it$ and $\widehat D_iR_i\gamma t$ avoids $D_r$. 
Moreover, $\gamma\notin Q_k$ implies
\begin{equation*}
\widehat D_kt\cap\widehat D_k\gamma t=\varnothing.
\end{equation*}
Set $b_0=t$ and $b_1=\gamma t$, and put exact copies of $q_k$ on $D_kb_0$ and $D_kb_1$.

For $i=k+1,\ldots,r$, choose $u_i^{(0)}$ and $u_i^{(1)}$ successively.
Before each choice let $W$ be the union of the halos of all core and guard domains already chosen, excluding $D_r$. 
Then
\[
|W|\le 2\sum_{m<i}|\widehat D_m|+|\widehat D_i|=\mu_i;
\]
the last term is present only when the second choice at level $i$ is made.
Apply Lemma~\ref{lem:finite_avoidance} with $D=\widehat D_i$, $M=\mu_i$, and $b=b_0$ or $b=b_1$, respectively. 
Choose $u_i^{(\nu)}\in R_ib_\nu$ so that $\widehat D_i u_i^{(\nu)}\cap W=\varnothing$. 
By the choice of $t$, this halo also avoids $D_r$. 
Thus all selected halos are pairwise disjoint, and in particular the actual guard domains are disjoint from one another and
from the two core domains. 
Put an exact copy of $q_i$ on  $D_iu_i^{(\nu)}$ and let $\mathcal J_r^{\mathrm{mix}}$ be the family of all core and guard domains.

For $k\le m<i$, write $u_m^{(\nu)}=\alpha_m b_\nu$ and $u_i^{(\nu)}=\alpha_i b_\nu$, 
where $\alpha_k=e$ and $\alpha_m\in R_m$ for $m>k$. 
Since $D_m\subseteq S_m\subseteq P_{i-1}\subseteq S_i$ and $R_m,R_i\subseteq S_i$, we have
\begin{equation}\label{eq:free_guard_distance}
D_mu_m^{(\nu)}(u_i^{(\nu)})^{-1}
\subseteq D_mR_mR_i^{-1}\subseteq S_i^3.
\end{equation}

Apply the SCP to $(\widehat D_r,U_{m_r})$ and obtain a finite $\widehat D_r$-separated set $E_r$ such that
\begin{equation}\label{eq:free_Er}
E_r^{-1}U_{m_r}=X,
\qquad \widehat D_r s\cap\widehat D_r s'=\varnothing
\quad(s\ne s',\ s,s'\in E_r).
\end{equation}
Let $A_r$ be the union of $D_r$ and all actual blocks in  $\mathcal J_r^{\mathrm{mix}}$. 
Since $G$ is infinite, choose
\begin{equation}\label{eq:free_tau}
\tau\notin\biggl(\bigcup_{s\in E_r}\widehat D_r s\biggr)^{-1}A_r.
\end{equation}
Then the halos $\widehat D_r s\tau$ are pairwise disjoint and avoid $A_r$.
On each $D_r s\tau$, $s\in E_r$, put an exact copy of $q_r$; call these blocks $\mathcal J_r^{\mathrm{sync}}$. 
For every $n\le r$, the set $T_r$ contains $g_{n+1}^{\pm1}$, so the old color constraints cannot be violated
across two synchronization blocks or between a synchronization block and an already assigned block.

Set
\begin{equation*}
B_{r+1}=D_r\cup\{g_1,\ldots,g_{r+1}\}
\cup\bigcup_{J\in\mathcal J_r^{\mathrm{mix}}}J
\cup\bigcup_{J\in\mathcal J_r^{\mathrm{sync}}}J,
\qquad D_{r+1}=P_rB_{r+1}.
\end{equation*}
All blocks listed in $B_{r+1}$ carry their prescribed exact patterns. 
Process $i=r,r-1,\ldots,1$. 
Just before level $i$, let $\mathcal H_i$ consist of $D_r$, all mixing and synchronization blocks, 
and all ordinary blocks already placed at levels greater than $i$. 
Define
\begin{equation*}
\Lambda_i=\{\ell:\widehat D_i\ell\subseteq V_i^{(r+1)},\quad
\widehat D_i\ell\cap J=\varnothing\text{ for every }J\in\mathcal H_i\}.
\end{equation*}
Choose an inclusion-maximal $M_i\subseteq\Lambda_i$ for which the halos $\widehat D_i t$, $t\in M_i$, are pairwise disjoint, 
and put exact copies of $q_i$ on $D_it$, $t\in M_i$. 
Since all members of $\mathcal H_i$ lie in $V_{i+1}^{(r+1)}$, 
every buffer domain in $V_i^{(r+1)}\setminus V_{i+1}^{(r+1)}$ avoids $\mathcal H_i$.

For $n\le r$, consider the graph $\mathcal G(g_{n+1},D_{r+1})$. 
The partial coloring supplied by the exact blocks is proper: inside a block this follows from the color invariant in Step~2, 
while an edge joining two distinct assigned blocks would have one endpoint in the halo of the other,
contrary to halo avoidance or pairwise halo disjointness. 
Apply Lemma~\ref{lem:three_color_extension} to extend this partial coloring to $D_{r+1}$. 
For the new coordinate $n=r+1$ start with the empty partial coloring and apply the same lemma. 
The resulting $q_{r+1}$ extends $q_r$, has the required color invariant, and all color coordinates are now defined on $D_{r+1}$.

\medskip
\noindent\textbf{Claim 1.}
For every $1\le i\le r$,
\begin{equation*}
\Fr_{i,r+1}\subseteq S_i^6\Occ_{q_{r+1}}(q_i).
\end{equation*}
\begin{proof}
The product geometry shows $V_i^{(r+1)}=S_i^5V_{i+1}^{(r+1)}\subseteq D_{r+1}\subseteq S_i^6V_{i+1}^{(r+1)}$. 
If $a\in\Fr_{i,r+1}$, Lemma~\ref{lem:buffer} with $R=6$, $D=\widehat D_i$, $C=\{e\}$, and $L=G$ shows $\ell$ with
$\widehat D_i\ell\subseteq V_i^{(r+1)}\setminus V_{i+1}^{(r+1)}$ and $a\in S_i^4\ell$. 
Consequently, $\ell\in\Lambda_i$. 
Lemma~\ref{lem:packing} shows $\ell\in S_i^2t$ for some $t\in M_i$, and the exact copy on $D_it$ makes $t\in\Occ_{q_{r+1}}(q_i)$. 
Thus $a\in S_i^6t$.
\end{proof}

\medskip
\noindent\textbf{Claim 2.}
For every $1\le i\le r$,
\begin{equation*}
D_{r+1}\subseteq S_i^8\Occ_{q_{r+1}}(q_i).
\end{equation*}
\begin{proof}
Fix $i$. 
Suppose $a\in D_j u$ belongs to a central, synchronization, mixing, or higher ordinary block carrying $q_j$. 
If $j>i$, the inclusion $D_j\subseteq S_i^8\Occ_{q_j}(q_i)$ from Step~2 implies that $a\in S_i^8vu$ for some $v\in\Occ_{q_j}(q_i)$, and occurrence inheritance shows that $vu\in\Occ_{q_{r+1}}(q_i)$. 
If $j=i$, the exact centre $u$ itself covers $a\in D_i u\subseteq S_i u$. 
If $j<i$, the block is a lower mixing block, and \eqref{eq:free_guard_distance} covers $a$ by the corresponding level-$i$ guard centre. 
A point in a current level-$i$ ordinary block is likewise covered by its exact centre.

It remains to consider a point $a$ outside the central, synchronization, and mixing blocks and all ordinary blocks of levels at least $i$.
Put $\ell=a$. 
For $y=d\ell$ with $d\in\widehat D_i\subseteq S_i$,
\begin{equation}\label{eq:free_candidate_distance}
ay^{-1}=d^{-1}\in S_i,
\qquad ya^{-1}=d\in S_i.
\end{equation}
If $\widehat D_i\ell\not\subseteq V_i^{(r+1)}$, choose $y\in\widehat D_i\ell\setminus V_i^{(r+1)}$. 
If $a\in S_i^3V_{i+1}^{(r+1)}$, then \eqref{eq:free_candidate_distance} would show
$y\in S_i^4V_{i+1}^{(r+1)}\subseteq V_i^{(r+1)}$, a contradiction.
Thus $a\in\Fr_{i,r+1}$ and Claim~1 applies.

Suppose $\widehat D_i\ell\subseteq V_i^{(r+1)}$ and this halo meets an exact higher block $D_ju$, $j>i$.
Write $y=d\ell=bu$ in the intersection, where $d\in\widehat D_i$ and $b\in D_j$.
Since $a$ is outside that block, $au^{-1}\notin D_j$, while $(au^{-1})b^{-1}=ay^{-1}\in S_i$. 
If $b\notin\Fr_{i,j}$, then $b\in S_i^3V_{i+1}^{(j)}$ and $au^{-1}\in S_i^4V_{i+1}^{(j)}\subseteq V_i^{(j)}\subseteq D_j$, a contradiction. Thus $b\in\Fr_{i,j}$. 
Since $\Fr_{i,j}\subseteq S_i^6\Occ_{q_j}(q_i)$ by Step~2, there is $v\in\Occ_{q_j}(q_i)$ with $b\in S_i^6v$, and inheritance shows $vu\in\Occ_{q_{r+1}}(q_i)$. 
Consequently, $a\in S_i^7vu\subseteq S_i^8vu$.

If the halo meets a fixed block $D_i u$ carrying $q_i$, choose $y=bu$ in the intersection with $b\in D_i\subseteq S_i$. Equation~\eqref{eq:free_candidate_distance} implies $a\in S_i y\subseteq S_i^2u$, and $u\in\Occ_{q_{r+1}}(q_i)$ provides the required covering. This includes central, synchronization, and mixing blocks of level $i$.

If the halo meets a fixed lower mixing block $D_mu_m^{(\nu)}$ with $m<i$,
then \eqref{eq:free_guard_distance} and \eqref{eq:free_candidate_distance} show $a\in S_i^4u_i^{(\nu)}$. 
In the remaining case the halo avoids every member of $\mathcal H_i$, so $\ell\in\Lambda_i$. Intersections with ordinary blocks inserted at level $i$ or below do not enter the definition of $\Lambda_i$.
Maximality and Lemma~\ref{lem:packing} show $\ell\in S_i^2t$ for some $t\in M_i\subseteq\Occ_{q_{r+1}}(q_i)$, and consequently $a\in S_i^2t$.
These cases exhaust $D_{r+1}$.
\end{proof}

At the end of the stage choose $R_{r+1}$ by Lemma~\ref{lem:finite_avoidance} with $D=\widehat D_{r+1}$ and $M=\mu_{r+1}$. 
Enlarge $S_{r+1}$ to contain all finite data used above and $D_{r+1}$. 
Then
\[
D_{r+1}\subseteq S_{r+1}^8\Occ_{q_{r+1}}(q_{r+1})
\]
because $e\in\Occ_{q_{r+1}}(q_{r+1})$. 
Thus the data satisfy the induction hypothesis in Step~2 at level $r+1$.
This completes the induction.

\medskip
\noindent\textbf{Step 4: Passage to the limit and verification of the conclusion.}
Since $g_r\in B_r\subseteq D_r$, the domains $D_r$ exhaust $G$. 
The compatibility relations define a unique $z\in\mathcal Z$. 
Put $Y=\overline{Gz}$ and, for each $i$, let
\[
[q_i]=\{w\in\mathcal Z:w_n(d)=c_{n,i}(d)\text{ for all }d\in D_i,
\ 1\le n\le i\},
\qquad W_i=Y\cap[q_i].
\]
The sets $W_i$ form a decreasing neighbourhood basis at $z$. 
Since $D_j\subseteq S_i^8\Occ_{q_j}(q_i)$ for every $j\ge i$ by Step~2 and the domains exhaust $G$, occurrence inheritance shows that
\begin{equation*}
G=S_i^8\Occ_z(q_i)=S_i^8N_G(z,W_i).
\end{equation*}
Thus $z$ and $Y$ are minimal.

For fixed $k$, every $\gamma\notin Q_k$ occurs in a mixing requirement. 
The corresponding two exact $q_k$-cores are preserved by all later extensions, so $\gamma\in N_G(W_k,W_k)$. 
Since $Q_k$ is finite, $N_G(W_k,W_k)$ is cofinite. 
For arbitrary non-empty open $U,V\subseteq Y$, choose $a,b\in G$ and $k$ with $aW_k\subseteq U$ and $bW_k\subseteq V$.
Then
\begin{equation*}
N_G(aW_k,bW_k)=bN_G(W_k,W_k)a^{-1}
\end{equation*}
is cofinite and contained in $N_G(U,V)$. 
Consequently, $Y$ is strongly mixing.

For every $n$ and every $h\in G$, choose $i$ with $h,g_{n+1}h\in D_i$.
The color constraints are closed and shift invariant, so they hold for every $y=(c_n)_{n\ge1}\in Y$. In particular, for every $n$ and $h\in G$, $c_n(g_{n+1}h)\ne c_n(h)$. If $g_{n+1}$ fixed a point $y\in Y$, 
then at $h=e$ the shift formula would show $c_n(g_{n+1})=c_n(e)$, a contradiction.
Thus every nonidentity element acts without fixed points and $Y$ is free.

When $(\ell_r,m_r)=(\ell,m)$ is processed, \eqref{eq:free_Er} and \eqref{eq:free_tau} show a set $E_r$ and a translate $\tau$ with
$E_r^{-1}U_m=X$. 
Put $y_{m,\ell}=\tau z$. Since each synchronization block is an exact copy of $q_r$ and $\ell\le r$,
\begin{equation*}
E_ry_{m,\ell}\subseteq W_\ell.
\end{equation*}
Lemma~\ref{lem:relative_sync} produces $X\times Y$ minimal and hence $X\perp Y$.

The space $Y$ is a closed subspace of the compact metrizable zero-dimensional space $\mathcal Z$. 
It is infinite because the action is free and $G$ is infinite. 
If a minimal compact flow had an isolated point, minimality would make  it finite and discrete; 
this is impossible for a free action of an infinite group.
Thus $Y$ has no isolated points and is a Cantor space.
\end{proof}

\begin{proposition}\label{prop:free_mixing_realization}
Let $e\in F\subseteq G$.
If there exists a topologically strongly mixing (or topologically weakly mixing of all finite orders) minimal $G$-flow $X$, a point $x\in X$, and an open neighbourhood $U$ of $x$ such that $ N_G(x,U)\subseteq F$,
then there exists a free, topologically strongly mixing (or weakly mixing of all finite orders) minimal $G$-flow on a Cantor space $X$, 
a point $x\in X$ , and a clopen neighbourhood $U$ of $x$ such that $N_G(x,U) \subseteq F$.
\end{proposition}

\begin{proof}
We first prove the topologically strongly mixing case.  At the end we verify the change needed for weak mixing of all finite orders.

By Theorem~\ref{thm:sm_sets}, there exists a point $x_0\in 2^G$ and 
a minimal topologically strongly mixing subshift $X_0=\overline{Gx_0}\subseteq 2^G$ such that $x_0(e)=1$ and
\[
N_G(x_0,U_0)\subseteq F,
\qquad
U_0:=\{z\in X_0:z(e)=1\}.
\]
In particular, $U_0$ is a clopen neighbourhood of $x_0$.

Apply Theorem~\ref{thm:free_strong_mixing_disjoint} to the compact metrizable minimal $G$-flow $X_0$. 
We obtain a free topologically strongly mixing minimal $G$-flow $Y$ on a Cantor space such that
$X_0\perp Y$.
Define $Z:=X_0\times Y$.
Since $X_0$ and $Y$ are minimal and disjoint, $Z$ is minimal.
Indeed, if $M\subseteq X_0\times Y$ is a minimal subsystem, then the coordinate projections of $M$ are non-empty closed invariant subsets
of $X_0$ and $Y$, respectively, and consequently are all of $X_0$ and $Y$. 
Thus $M$ is a joining of $X_0$ and $Y$. 
Disjointness forces $M=X_0\times Y$.
Now we will show that $Z$ is free.
If $g(x,y)=(x,y)$ for some $(x,y)\in Z$, then $gy=y$, and the freeness of $Y$ produces $g=e$.
It remains to verify that $Z$ is topologically strongly mixing. Let
$O_1,O_2\subseteq Z$ be non-empty open sets. Choose non-empty open rectangles
\[
A_1\times B_1\subseteq O_1, \qquad A_2\times B_2\subseteq O_2.
\]
Then
\[
N_G(A_1\times B_1,A_2\times B_2)
 =N_G(A_1,A_2)\cap N_G(B_1,B_2).
\]
Both sets on the right are cofinite because $X_0$ and $Y$ are topologically strongly mixing. Consequently, their intersection is cofinite. Since
\[
N_G(A_1\times B_1,A_2\times B_2)
 \subseteq N_G(O_1,O_2),
\]
$N_G(O_1,O_2)$ is also cofinite. Therefore $Z$ is topologically strongly mixing.

Fix $y_0\in Y$, and put $z=(x_0,y_0)$ and $W=U_0\times Y$.
Then $W$ is a clopen neighbourhood of $z$, and
\[
\begin{aligned}
N_G(z,W)
&=\{g\in G:(gx_0,gy_0)\in U_0\times Y\}\\
&=\{g\in G:gx_0\in U_0\}=N_G(x_0,U_0)\subseteq F.
\end{aligned}
\]

We show that $Z$ is a Cantor space. 
Indeed, $X_0\subseteq 2^G$ is compact metrizable and zero-dimensional, while $Y$ is a Cantor space.
Consequently $Z=X_0\times Y$ is compact, metrizable, and zero-dimensional. 
It has no isolated points, since every basic neighbourhood $A\times B$ of a point $(x,y)$ contains
$(x,y')\neq(x,y)$ for some $y'\in B$, because $Y$ has no isolated points. 
Thus $Z$ is a Cantor space. 

For weak mixing of all finite orders, choose the minimal binary subshift $X_0$, the point $x_0$, and the identity cylinder $U_0$ using Theorem~\ref{thm:wm_allorders}(3). 
Apply Theorem~\ref{thm:free_strong_mixing_disjoint} to obtain a free strongly mixing Cantor flow $Y$ disjoint from $X_0$. Disjointness and minimality of the factors make $Z=X_0\times Y$ minimal, while freeness follows from the $Y$-coordinate. 
The product is compact metrizable and zero-dimensional, and it has no isolated points because $Y$ has none.
For a non-empty open rectangle $A\times B\subseteq Z$, we have
\[
N_G(A\times B,A\times B)=N_G(A,A)\cap N_G(B,B).
\]
Lemma~\ref{lem:allorders_self_hitting} makes $N_G(A,A)$ bitranslation thick, and strong mixing makes $N_G(B,B)$ cofinite.
Their intersection is bitranslation thick by Lemma~\ref{lem:bt_finite_deletion}. 
Every non-empty open subset of $Z$ contains such a rectangle, so its self-hitting set is bitranslation thick as well. 
Since $Z$ is minimal and hence topologically transitive, Lemma~\ref{lem:allorders_self_hitting}(3) proves weak mixing of all finite orders. 
For any $y_0\in Y$, the clopen neighbourhood $W=U_0\times Y$ of $(x_0,y_0)$ satisfies $N_G((x_0,y_0),W)=N_G(x_0,U_0)\subseteq F$.
\end{proof}

Recall that a subset $F$ of $G$ is right thickly syndetic 
if $\Int_K(F)$ is right syndetic for any finite subset $K$ of $G$. 
If $e\in F\subset G$ is right thickly syndetic,
then $\Int_{K\cup Kg}(F)$ is right syndetic for every finite $K\subseteq F$ and every $g\in G$. 
Thus $\Gamma_F(K)=G$, and Lemma~\ref{lem:interior_vs_product} implies $\Int_K(F)\Int_K(F)^{-1}=G$. 
Assertion (2) of Theorem~\ref{thm:sm_sets} therefore holds for $F$.
Therefore, we have the following consequence.

\begin{corollary}\label{coro:rts_mixing_realization}
 If $e\in F\subset G$ is right thickly syndetic, then  
\begin{enumerate}
    \item there exists $e\in A\subseteq F$ such that $\one_A$ is a minimal point of $2^G$ and $\overline{G\one_A}$ is topologically strongly mixing;
    \item there exists a free, topologically strongly mixing minimal $G$-flow on a Cantor space $X$, a point $x\in X$, and a clopen neighbourhood $U$ of $x$ such that $N_G(x,U) \subseteq F$.
\end{enumerate}
\end{corollary}

\section{Applications to dynamical largeness and disjointness}
\label{sec:disjointness_applications}

This section applies the realization theorems of Sections~\ref{sec:proof} and \ref{sec:mixing_characterizations} to dynamically syndetic sets, dynamically thick sets, sets of pointwise recurrence, and disjointness.

\subsection{Dynamically syndetic sets} 
A set $F\subseteq G$ is \emph{dynamically syndetic} if
$N_G(x,U)\subseteq F$ for some minimal $G$-flow $X$, some point $x\in X$, and some non-empty open set $U\subset X$.  
It is \emph{dynamically central syndetic} if the witness can be chosen with $x\in U$.

\begin{lemma}\label{lem:dynamic_syndetic_translate}
	For $F\subseteq G$, the following are equivalent:
	\begin{enumerate}
		\item $F$ is dynamically syndetic;
		\item $Fg^{-1}$ is dynamically central syndetic for some $g\in G$;
		\item $F=Cg$ for some dynamically central syndetic $C$ and some $g\in G$.
	\end{enumerate}
\end{lemma}

\begin{proof}
	If $N_G(x,U)\subseteq F$, choose $g\in G$ with $gx\in U$.  
	Then
	\[
	N_G(x,U)g^{-1}=N_G(gx,U)\subseteq Fg^{-1}.
	\]
	Conversely, if $N_G(y,V)\subseteq Fg^{-1}$ and $y\in V$, then
	\[
	N_G(y,V)g=N_G(g^{-1}y,V)\subseteq F. \qedhere 
	\]
\end{proof}

\begin{proposition}\label{prop:dynamic_syndetic_interior}
A set $F\subseteq G$ is dynamically syndetic if and only if there is a non-empty $A\subseteq F$ such that
$\Int_K(A)$ is right syndetic for every finite $K\subseteq A$.
\end{proposition}

\begin{proof}
Suppose first that $F$ is dynamically syndetic. Choose a witness $N_G(x,U)\subseteq F$ and choose $g\in G$ with $gx\in U$. Then $E:=Fg^{-1}$ is dynamically central syndetic by Lemma~\ref{lem:dynamic_syndetic_translate}. By Theorem~\ref{thm:equivalent_characterizations}, there is a non-empty $B\subseteq E$ such that $\Int_L(B)$ is right syndetic for every finite $L\subseteq B$. Put $A:=Bg\subseteq F$. For finite $K\subseteq A$, let $L:=Kg^{-1}\subseteq B$. The identity
\[
\Int_K(A)=g^{-1}\Int_L(B)g
\]
shows that $\Int_K(A)$ is right syndetic, since right syndeticity is invariant under conjugation.

Conversely, suppose that $A\subseteq F$ is non-empty and that every $\Int_K(A)$ is right syndetic for finite $K\subseteq A$. Choose $d\in A$ and put $E:=Ad^{-1}$. For finite $L\subseteq E$, set $K:=Ld\subseteq A$. Then
\[
\Int_L(E)=d^{-1}\Int_K(A)d.
\]
Thus $E$ is dynamically central syndetic by Theorem~\ref{thm:equivalent_characterizations}; Lemma~\ref{lem:dynamic_syndetic_translate} then shows that $F$ is dynamically syndetic.
\end{proof}

\subsection{Dynamically thick sets}

A set $F\subseteq G$ is \emph{dynamically thick} if
\[
F\cap N_G(x,U)\neq\varnothing
\]
for every minimal flow $X$, every $x\in X$, and every non-empty open $U\subset X$.

Dynamically thick sets are upward hereditary. Their duality with dynamically syndetic sets will turn the finite-interior characterization into a complement criterion.

\begin{lemma}\label{lem:dynamic_thick_dual}
	For $F\subseteq G$, the following are equivalent:
	\begin{enumerate}
		\item $F$ is dynamically thick;
		\item $F$ intersects every dynamically syndetic subset of $G$;
		\item $G\setminus F$ is not dynamically syndetic.
	\end{enumerate}
\end{lemma}

Dynamically thick sets were called dense orbit sets in~\cite{GTWZ21} by Glasner et al. 
They also asked the following question:

\begin{question}[{\cite[Question 9.6]{GTWZ21}}]
	Characterize the dynamically thick sets in countable, discrete groups.
\end{question} 

Glasscock and Le provided combinatorial characterizations of dynamically thick sets in~\cite[Theorems A and B]{GL26} for $\mathbb{N}$-actions.
We generalize the characterizations to countably infinite group actions.

\begin{theorem}\label{thm:dynamic_thick_characterization}
	For $F\subseteq G$, the following are equivalent:
	\begin{enumerate}
		\item $F$ is dynamically thick;
		\item for every $F\subseteq B\subsetneq G$, there is a finite
		$K\subseteq G\setminus B$ such that $K^{-1}B$ is right thick.
	\end{enumerate}
\end{theorem}

\begin{proof}
	Let $C=G\setminus F$.  
	By Lemma~\ref{lem:dynamic_thick_dual} and Proposition~\ref{prop:dynamic_syndetic_interior}, dynamical thickness is equivalent to the assertion that every non-empty $D\subseteq C$ has a finite $K\subseteq D$ for which $\Int_K(D)$ is not right syndetic.  
	Given $F\subseteq B\subsetneq G$, take $D=G\setminus B$. 
	The complement identity $G\setminus\Int_K(D)=K^{-1}B$ and the equivalence between failure of right syndeticity and thickness of the complement show (2).  
	
	The converse follows by taking $B=G\setminus D$ whenever the finite-interior condition fails.
\end{proof}

Applying Theorem~\ref{thm:dynamic_thick_characterization}, we have the following immediate consequence.
\begin{corollary}
Every dynamically thick subset of $G$ is piecewise syndetic.
\end{corollary}

\subsection{Sets of pointwise recurrence}

A set $F\subseteq G$ is a \emph{set of pointwise recurrence} if
\[
F\cap\bigl(N_G(x,U)\setminus\{e\}\bigr)\neq\varnothing
\]
for every minimal $G$-flow $X$, every $x\in X$, and every neighbourhood $U$ of $x$.
Equivalently, $F$ intersects $S\setminus\{e\}$ for every dynamically central syndetic set $S$. 
Thus membership of $e$ in $F$ has no effect on whether $F$ is a set of pointwise recurrence. 
In the structural results below, we take the recurrence sets to be subsets of $G\setminus\{e\}$.

\begin{theorem}
	\label{thm:pointwise_recurrence_characterization}
	Let $A\subseteq G\setminus\{e\}$ and $C=G\setminus A$. 
	The following are equivalent:
	\begin{enumerate}
		\item $A$ is a set of pointwise recurrence;
		\item for every $D$ with $e\in D\subseteq C$, there is a finite
		$K\subseteq D$ such that $\Int_K(D)$ is not right syndetic;
		\item for every $B\supseteq A$ with $e\notin B$, there is a finite
		$K\subseteq G\setminus B$ such that $B\cup K^{-1}B$ is right thick.
	\end{enumerate}
\end{theorem}

\begin{proof}
	The equivalence of (1) and (2) follows from Theorem~\ref{thm:equivalent_characterizations}(1)$\Leftrightarrow$(2) applied to the complement $C$.
	Indeed, failure of (1) gives a minimal flow, a point $x$, and a neighbourhood $U$ of $x$ such that
	\[
	A\cap\bigl(N_G(x,U)\setminus\{e\}\bigr)=\varnothing.
	\]
	Since $A\subseteq G\setminus\{e\}$, this says $N_G(x,U)\subseteq C$. 
	Because $U$ is a neighbourhood of $x$, we have $e\in N_G(x,U)$, so $C$ is dynamically central syndetic. 
	Conversely, if $C$ is dynamically central syndetic, it contains a return-time set $N_G(x,U)$ with $x\in U$, and then $A$ misses $N_G(x,U)\setminus\{e\}$.
	
	For (2)$\implies$(3), take $D=G\setminus B$. 
	Then $e\in D\subseteq C$, and the complement identity
	\[
	G\setminus\Int_K(D)=K^{-1}B
	\]
	together with the equivalence between failure of right syndeticity and thickness of the complement shows that $K^{-1}B$, 
	and consequently $B\cup K^{-1}B$, is right thick.
	
	Conversely, if (2) fails, choose $e\in D\subseteq C$ such that $\Int_K(D)$ is right syndetic for every finite $K\subseteq D$. 
	Put $B=G\setminus D$. Then $A\subseteq B$, $e\notin B$, and
	\[
	G\setminus(B\cup K^{-1}B)=D\cap\Int_K(D)=\Int_{K\cup\{e\}}(D)
	\]
	is right syndetic for every finite $K\subseteq D$. Hence $B\cup K^{-1}B$ is not right thick, contradicting (3).
\end{proof}

\begin{proposition}
	Every set of pointwise recurrence $A\subseteq G\setminus\{e\}$ is right piecewise syndetic.
\end{proposition}
\begin{proof}
	Let $A\subseteq G\setminus\{e\}$ be a set of pointwise recurrence. 
	For any right thickly syndetic set $B\subseteq G$, the set $B\cup\{e\}$ is right thickly syndetic and contains $e$. 
	By Corollary~\ref{coro:rts_mixing_realization}, there exist a minimal $G$-flow $X$, a point $x\in X$, and a neighbourhood $U$ of $x$ such that
	\[
	N_G(x,U)\subseteq B\cup\{e\}.
	\]
	Pointwise recurrence gives
	\[
	A\cap\bigl(N_G(x,U)\setminus\{e\}\bigr)\neq\varnothing.
	\]
	Since $e\notin A$, it follows that $A\cap B\neq\varnothing$. Thus $A$ is right piecewise syndetic.
\end{proof}

Recall that a collection is partition regular if, whenever $A$ belongs to the collection and $A=B\cup C$, at least one of $B$ and $C$ belongs to the collection.
In~\cite{HKM16}, Host, Kra and Maass asked the following question:

\begin{question}[{\cite[Question 6.5]{HKM16}}]
	Is the family of sets of pointwise recurrence partition regular?
\end{question}

Glasscock and Le gave a negative answer to this question in~\cite[Theorem D]{GL24} for $\mathbb{N}$-actions. We show that this question has a negative answer for all countably infinite group actions.

\begin{theorem}\label{thm:split_pointwise_recurrence}
	Every set of pointwise recurrence $A\subseteq G\setminus\{e\}$ has a partition
	$A=A_1\sqcup A_2$ in which neither $A_1$ nor $A_2$ is a set of pointwise recurrence.
\end{theorem}

\begin{proof}
	Apply Corollary~\ref{cor:two_piece_dcs} to $G$, which is dynamically central syndetic, as witnessed by the one-point flow. 
	Choose disjoint $F_1,F_2\subseteq G$ such that $F_1\cup\{e\}$ and $F_2\cup\{e\}$ are dynamically central syndetic. 
	Put $A_1=A\cap F_1$ and $A_2=A\setminus F_1$. 
	Then $A=A_1\sqcup A_2$, and
	\[
	A_1\cap F_2=\varnothing\quad\text{and}\quad A_2\cap F_1=\varnothing.
	\]
	Because $e\notin A_1\cup A_2$, each piece misses the punctured part of one of the dynamically central syndetic sets $F_2\cup\{e\}$ and $F_1\cup\{e\}$. 
	Neither piece is therefore a set of pointwise recurrence.
\end{proof}

\begin{remark}\label{rem:finite_deletion_largeness}
	The finite-deletion conclusion of Corollary~\ref{cor:finite_deletion_dcs} extends to dynamically syndetic sets: if $F\subseteq G$ is dynamically syndetic and $K\subseteq G$ is finite, then $F\setminus K$ is dynamically syndetic.
	Indeed, choose a witnessing return-time set $N_G(x,U)\subseteq F$ and then choose $g\in N_G(x,U)\setminus K$, which is possible because return-time sets in minimal flows are syndetic and hence infinite. 
	By Lemma~\ref{lem:dynamic_syndetic_translate}, $Fg^{-1}$ is dynamically central syndetic. 
	Since $e\notin Kg^{-1}$, Corollary~\ref{cor:finite_deletion_dcs} shows that
	\[
	(F\setminus K)g^{-1}=Fg^{-1}\setminus Kg^{-1}
	\]
	is dynamically central syndetic, and Lemma~\ref{lem:dynamic_syndetic_translate} gives the claim.
	
	Consequently, dynamically thick sets are also preserved under finite deletion. 
	If $F$ is dynamically thick, then for every dynamically syndetic $D$ the set $D\setminus K$ is dynamically syndetic, so $F\cap(D\setminus K)\neq\varnothing$. 
	Hence $F\setminus K$ intersects every dynamically syndetic set and is dynamically thick by Lemma~\ref{lem:dynamic_thick_dual}.
	
	Finally, if $A$ is a set of pointwise recurrence and $K\subseteq G$ is finite, then $A\setminus K$ is still a set of pointwise recurrence. 
	For any dynamically central syndetic set $S$, apply Corollary~\ref{cor:finite_deletion_dcs} to $K\setminus\{e\}$ to see that $S\setminus(K\setminus\{e\})$ remains dynamically central syndetic. 
	Pointwise recurrence of $A$ therefore gives some
	\[
	a\in A\cap\bigl(S\setminus(K\setminus\{e\})\bigr)\setminus\{e\},
	\]
	and this $a$ belongs to $(A\setminus K)\cap(S\setminus\{e\})$. 
	Thus $A\setminus K$ is pointwise recurrent.
\end{remark}

\subsection{Disjointness}
The preceding recurrence tests concern intersections with return-time sets of minimal points. We now use simultaneous returns in product flows to express disjointness through these largeness conditions.

\begin{lemma}\label{lem:product_return_disjointness}
	Let $X$ and $Y$ be $G$-flows.  
	Assume that the flows are disjoint, and let $x\in X$ and $y\in Y$ be transitive points.
	Then
	\[
	\overline{G(x,y)}=X\times Y.
	\]
	Consequently, for every non-empty open $U\subseteq X$ and $V\subseteq Y$,
	\begin{equation*}
		N_G(x,U)\cap N_G(y,V)\neq\varnothing.
	\end{equation*}
\end{lemma}

\begin{proof}
	The set $J=\overline{G(x,y)}$ is a non-empty closed invariant subset of 	$X\times Y$.  
	Transitivity of $x$ and $y$ makes both coordinate projections of $J$ onto, so $J$ is a joining.  
	Disjointness produces $J=X\times Y$.
	Meeting $U\times V$ is exactly the displayed intersection.
\end{proof}

Using Lemma~\ref{lem:product_return_disjointness}, we characterize when the flow $X$ containing a transitive point is disjoint from every minimal flow in terms of its return-time sets.

\begin{proposition}\label{prop:dcs_disjointness_test}
	Let $X$ be a $G$-flow and let $x\in X$ be a transitive point.
	Then the following assertions are equivalent:
	\begin{enumerate}
		\item $X$ is disjoint from any minimal 
		$G$-flow.
		\item $N_G(x,U)$ is a set of pointwise recurrence for any non-empty open $U\subset X$.
		\item $N_G(x,U)$ is dynamically thick for any neighbourhood $U$ of $x$. 
	\end{enumerate}
\end{proposition}

\begin{proof}
(1)$\implies$(2). 
Let $U\subseteq X$ be non-empty and open, let $Y$ be a minimal $G$-flow, and let $y\in Y$ and $V$ be a neighbourhood of $y$. 
Since $X$ and $Y$ are disjoint, Lemma~\ref{lem:product_return_disjointness} implies that the orbit of $(x,y)$ is dense in $X\times Y$. 
We show that this orbit meets $U\times V$ at a point other than $(x,y)$.

If $x\notin U$ or $y\notin V$, every orbit point in $U\times V$ is already different from $(x,y)$. 
Suppose now that $x\in U$ and $y\in V$. 
If $x$ is not isolated in $X$, choose a non-empty open set $U'\subseteq U$ with $x\notin U'$. 
Density gives $t\in G$ such that $(tx,ty)\in U'\times V$, and then $t\ne e$. 
If $x$ is isolated and $y$ is not isolated, choose a non-empty open set $V'\subseteq V$ with $y\notin V'$. Density applied to the open set $\{x\}\times V'$ gives $t\in G$ with $tx=x$ and $ty\in V'$, so again $t\ne e$.

It remains to consider the case where both $x$ and $y$ are isolated. 
Put $H=\{g\in G:gx=x\}$. 
By the free minimal Cantor flow furnished by Corollary~\ref{coro:rts_mixing_realization}, disjointness implies that $H$ has a dense orbit in an infinite minimal flow; 
hence $H$ is infinite. 
Applying density of the orbit of $(x,y)$ to the open set $\{x\}\times\{y\}$ shows that the $H$-orbit of $y$ is dense in $Y$. 
Since $y$ is isolated and $Y$ is minimal, $Y$ is finite, so this orbit is all of $Y$. 
The stabilizer of $y$ in $H$ therefore has finite index in the infinite group $H$, and contains some $t\ne e$.
In all cases we have found a non-identity $t\in N_G(x,U)\cap N_G(y,V)$. 
As $Y,y,V$ were arbitrary, $N_G(x,U)$ is a set of pointwise recurrence.

(2)$\implies$(3). 
Let $U$ be a neighbourhood of $x$, and let $D\subseteq G$ be dynamically syndetic. 
There are a minimal $G$-flow $Y$, a point $y_0\in Y$, and a non-empty open set $V\subseteq Y$ such that
\[
N_G(y_0,V)\subseteq D.
\]
By minimality choose $a\in G$ with $ay_0\in V$. Then $a^{-1}V$ is a neighbourhood of $y_0$, and $a^{-1}U$ is a non-empty open subset of $X$. 
By (2), $N_G(x,a^{-1}U)$ is a set of pointwise recurrence. 
Hence there is
\[
t\in N_G(x,a^{-1}U)\cap\bigl(N_G(y_0,a^{-1}V)\setminus\{e\}\bigr).
\]
It follows that $(at)x\in U$ and $(at)y_0\in V$, so
\[
at\in N_G(x,U)\cap N_G(y_0,V)\subseteq N_G(x,U)\cap D.
\]
Thus $N_G(x,U)$ intersects every dynamically syndetic set. 
Lemma~\ref{lem:dynamic_thick_dual} now implies that $N_G(x,U)$ is dynamically thick.

(3)$\implies$(1). 
Let $Y$ be any minimal $G$-flow, and let $J$ be a joining of $X$ and $Y$. 
Since $x$ is transitive, there is $y\in Y$ with $(x,y)\in J$.
We show that $(x,y)$ is transitive in $X\times Y$.
Given non-empty open sets $U\subseteq X$ and $V\subseteq Y$, choose $h\in G$ with $hx\in U$. 
Then $h^{-1}U$ is a neighbourhood of $x$, so by (3) the set $N_G(x,h^{-1}U)$ is dynamically thick. 
Since $h^{-1}V$ is non-empty and open, there is
\[
t\in N_G(x,h^{-1}U)\cap N_G(y,h^{-1}V).
\]
Consequently $(ht)x\in U$ and $(ht)y\in V$. 
Thus the orbit of $(x,y)$ is dense in $X\times Y$. 
As $J$ is closed and invariant and contains $(x,y)$, we obtain $J=X\times Y$. 
Therefore $X$ and $Y$ are disjoint.
\end{proof}

The preceding realization results in Sections 3 and 4 have the following consequence.
\begin{proposition}
	Let $X$ be a $G$-flow with a transitive point. 
	Let $\mathcal C$ be one of the following classes: all minimal $G$-flows; 
	minimal $G$-flows that are topologically weakly mixing of all finite orders; 
	or topologically strongly mixing minimal $G$-flows.
	The following assertions are equivalent:
	\begin{enumerate}[label=\textup{(\arabic*)}]
		\item $X$ is disjoint from every flow in $\mathcal C$;
		\item $X$ is disjoint from every minimal subflow in $\mathcal C$ of the Bernoulli shift $2^G$;
		\item $X$ is disjoint from every free minimal flow in $\mathcal C$ on a Cantor space.
	\end{enumerate}
\end{proposition}
\begin{proof}
(1)$\implies$(2) and (3). Both test classes are subclasses of $\mathcal C$.

(2)$\implies$(1) and (3)$\implies$(1). 
Fix a transitive point $x\in X$. 
Let $Y\in\mathcal C$ and let $J$ be a joining of $X$ and $Y$. Choose $y\in Y$ with $(x,y)\in J$. 
We show that $G(x,y)$ meets every non-empty open rectangle $U\times V\subseteq X\times Y$.
By minimality of $Y$, choose $a\in G$ with $ay\in V$.
Thus $a^{-1}V$ is a neighbourhood of $y$, and $D:=N_G(y,a^{-1}V)$ is a pointed return-time set of a flow in $\mathcal C$.

For the three choices of $\mathcal C$, Theorems~\ref{thm:equivalent_characterizations}, \ref{thm:wm_allorders}, and~\ref{thm:sm_sets}, respectively, realize a neighbourhood return-time set contained in $D$ in a minimal binary subshift belonging to $\mathcal C$. 
Under (2), use this realization and denote its point and neighbourhood by $z$ and $W$.
Under (3), use instead the free Cantor realization from Proposition~\ref{prop:more_equi_dyn_syn_sets}(4) for the class of all minimal flows, or Proposition~\ref{prop:free_mixing_realization} for either mixing class. 
Again write $z$ and $W$ for the realizing point and neighbourhood, with $N_G(z,W)\subseteq D$.
In either case the chosen test flow is disjoint from $X$. 
Lemma~\ref{lem:product_return_disjointness} implies
\[
N_G(x,a^{-1}U)\cap N_G(z,W)\ne\varnothing.
\]
Choose $t$ in this intersection. 
Then $tx\in a^{-1}U$ and $ty\in a^{-1}V$, so $(at)(x,y)\in U\times V$. 
Hence $G(x,y)$ is dense in $X\times Y$. 
Since $(x,y)\in J$, it follows that $J=X\times Y$. 
This proves (1) for each class $\mathcal C$.
\end{proof}

\begin{proposition}\label{prop:disjoint_strong_mixing} 
	Let $X$ be a compact Hausdorff $G$-flow and let $x\in X$ be a transitive point.
	If $X$ is disjoint from every minimal topologically strongly mixing  $G$-flow, then for every non-empty open $U\subseteq X$,
	$N_G(x,U)$ is right piecewise syndetic.  
	Consequently, the collection of minimal points of $X$ is dense.   
\end{proposition}

\begin{proof}
	Suppose that $R:=N_G(x,U)$ is not right piecewise syndetic.  
	A set is right piecewise syndetic precisely when it meets every right thickly syndetic set, see e.g.~\cite[Lemma~2.2]{XY24}. 
	Thus there is a right thickly syndetic set $T$ with $R\cap T=\varnothing$. 
	Choose $a\in T$ and put $E:=Ta^{-1}$, so $e\in E$.
	Corollary~\ref{coro:rts_mixing_realization} produces a minimal, topologically strongly mixing compact metrizable $G$-flow $(Y,G)$, a point $z\in Y$, and a non-empty clopen set $V\subseteq Y$ with $N_G(z,V)\subseteq E$. 
	Set $y:=a^{-1}z$. 
	Then $N_G(y,V)=N_G(z,V)a\subseteq T$.  
	The hypothesis and Lemma~\ref{lem:product_return_disjointness} imply
	\[
	\varnothing\neq N_G(x,U)\cap N_G(y,V)\subseteq R\cap T,
	\]
	a contradiction.  
    
	To prove density of minimal points in the compact Hausdorff setting, fix a non-empty open set $O\subseteq X$ and choose a non-empty open $V$ with $\overline V\subseteq O$. 
	Since $N_G(x,V)$ is right piecewise syndetic, there is a finite non-empty $K\subseteq G$ such that $K^{-1}N_G(x,V)$ is right thick. 
	Put $C=\bigcup_{k\in K}k^{-1}\overline V$.
	For every finite $E\subseteq G$, thickness provides $h\in G$ with $Eh\subseteq K^{-1}N_G(x,V)$, so $Ehx\subseteq C$. 
	The closed sets $\bigcap_{g\in E}g^{-1}C$ therefore have the finite-intersection property. Compactness provides a point $z$ with $Gz\subseteq C$. 
	Choose a minimal subflow $M\subseteq\overline{Gz}\subseteq C$. 
	For any $m\in M$, some $k\in K$ satisfies $km\in\overline V\subseteq O$. 
	Thus $O$ contains a minimal point. 
\end{proof}

\begin{remark}
	If $X$ is a non-trivial equicontinuous minimal $G$-flow, then $X$ is disjoint from every minimal $G$-flow $Y$ that is topologically weakly mixing of all finite orders. 
	This result is implicit in Chapter 11 of~\cite{A88}.
	Here we present a direct argument. 
	Let $J\subseteq X\times Y$ be a minimal joining, and suppose that $J$ misses a non-empty open rectangle $U\times V$. 
	Choose $x_0\in U$ and a symmetric entourage $E$ of $X$ with $E[x_0]\subseteq U$. 
	Equicontinuity and compactness provide a finite open cover $W_1,\ldots,W_n$ of $X$ such that $(gx,gx')\in E$ whenever $x,x'\in W_j$ and $g\in G$.

The projection from $J$ onto $Y$ sends every non-empty relatively open set to a set with non-empty interior.
Indeed, choose a non-empty relatively open set with closure inside the given set; finitely many translates of its closure cover $J$ by minimality. 
Their projections form a finite closed cover of $Y$, so at least one has non-empty interior. 
Thus for each $j$ there is a non-empty open $V_j\subseteq Y$ contained in the projection of $J\cap(W_j\times Y)$.
Transitivity of $Y^n$ provides $g\in G$ and $y_j\in V_j$ with $gy_j\in V$ for all $j$. Choose $x_j\in W_j$ with $(x_j,y_j)\in J$. 
Since $J$ misses $U\times V$, all $gx_j$ lie outside $U$. 
But $g^{-1}x_0\in W_j$ for some $j$, so $(gx_j,x_0)\in E$ and $gx_j\in U$, a contradiction. Hence every minimal joining is the full product, which proves disjointness. 
On the other hand, $X$ is not disjoint from every minimal $G$-flow, since $X$ itself is a non-trivial minimal $G$-flow. 
\end{remark}

\begin{remark}
	A homeomorphism $T$ of a compact metric space $(X,d)$ is uniformly rigid if there is a sequence $n_j\to\infty$ with $\sup_{x\in X}d(T^{n_j}x,x)\to0$.
	By Glasner and Maon~\cite[Proposition~6.5]{GM89}, there exists a homeomorphism $T$ on the two-dimensional torus $\mathbb{T}^2$ such that the cascade $(\mathbb{T}^2,T)$ is minimal, topologically weakly mixing, and uniformly rigid.
	Huang and Ye showed, in particular, that every topologically strongly mixing cascade is disjoint from every minimal uniformly rigid cascade~\cite[Theorem 4.4]{HY04}.
	Then $(\mathbb{T}^2,T)$ is disjoint from any topologically strongly mixing minimal $\mathbb{Z}$-action.
	On the other hand, $(\mathbb{T}^2,T)$ is not disjoint from every minimal $\mathbb{Z}$-action that is topologically weakly mixing of all finite orders, since it is itself such a system.
\end{remark}

\begin{remark}
	For $\mathbb Z$-actions, structural characterizations of disjointness from all minimal systems were obtained in~\cite{HSY20,HSXY25}. In particular, Huang, Shao, Xu and Ye~\cite[Theorem~A]{HSXY25} proved that a compact metrizable $\mathbb Z$-flow $X$ is disjoint from every minimal flow if and only if it contains a sequence of minimal subflows $M_i$ such that $\overline{\bigcup_iM_i}=X$ and $X\perp M_i$ for every $i$. This suggests asking whether the same equivalence holds for compact metrizable $G$-flows.
\end{remark}

\begin{remark}
	A quasi-factor of a compact flow $X$ is a minimal subflow of the space of non-empty closed subsets of $X$, equipped with the Vietoris topology and the induced $G$-action. 
	Suppose that every minimal $G$-flow admits an invariant probability measure. 
	Then~\cite[Corollary 11.28]{A88} states that a minimal flow $X$ is disjoint from every topologically weakly mixing minimal $G$-flow if and only if every non-trivial quasi-factor of $X$ has a non-trivial distal factor.
	A further question is to find structural characterizations, analogous to the dense-union criterion above, of $G$-flows disjoint from every topologically strongly mixing $G$-flow or from every minimal $G$-flow that is topologically weakly mixing of all finite orders. 
	In the strong-mixing case, the test class here also includes nonminimal flows, whereas Proposition~\ref{prop:disjoint_strong_mixing} concerns minimal test flows.
\end{remark}

\medskip 
\noindent \textbf{AI-use disclosure}:
The central ideas of this paper originated with the authors, who also established the results for finitely generated group actions. 
AI-assisted tools were used to assist in extending those results to arbitrary countably infinite group actions, to conduct literature searches, and to improve the exposition. 
The authors verified all mathematical arguments and take full responsibility for the final content.

\medskip 
\noindent \textbf{Acknowledgments}: 
The authors were partially supported by the National Key R\&D Program of China (2024YFA1013601), National Natural Science Foundation of China (12222110) and a grant from the Guangdong Provincial Department of Education (2025KCXTD013).

\end{document}